\documentclass[11pt]{amsart} 
\usepackage{amsmath,amsthm,amsfonts,amssymb,amscd,eucal}
\usepackage[all]{xy}
\usepackage{color}
\usepackage{cite,hyperref}
\usepackage{mathbbol}
\usepackage{latexsym}
\usepackage{caption}
\usepackage{appendix}

\usepackage[normalem]{ulem}

\usepackage{tikz}
\usetikzlibrary{decorations.markings}
\usetikzlibrary{decorations.pathreplacing}
\usetikzlibrary{arrows,shapes,positioning}

\tikzstyle directed=[postaction={decorate,decoration={markings,
    mark=at position #1 with {\arrow{>}}}}]
\tikzstyle rdirected=[postaction={decorate,decoration={markings,
    mark=at position #1 with {\arrow{<}}}}]
\usepackage{youngtab}
\usepackage{ytableau}
\usepackage[margin=1in]{geometry}
\usepackage{tensor,enumitem,verbatim}

\numberwithin{equation}{section}
\newtheorem{theorem}{Theorem}[section]
\newtheorem{proposition}[theorem]{Proposition}
\newtheorem{corollary}[theorem]{Corollary}
\newtheorem{lemma}[theorem]{Lemma}

\newtheorem{conjecture}[theorem]{Conjecture} 
\theoremstyle{definition}

\newtheorem{example}[theorem]{Example}

\newtheorem{remark}[theorem]{Remark}

\newcommand{\R}{\check{\mathsf{R}}}

\newcommand{\Af}{\mathsf{A}}
\newcommand{\Df}{\mathsf{D}}

\newcommand{\Hf}{\mathsf{H}}
\newcommand{\Ir}{\mathrm{I}}

\newcommand{\Pf}{\mathsf{P}}

\newcommand{\Sf}{\mathsf{S}}
\newcommand{\sfs}{\mathsf{s}}
\newcommand{\sft}{\mathsf{t}}

\newcommand{\dimm}{\mathsf{dim}}
\newcommand{\UU}{\mathsf{U}}
\newcommand{\Uq}{\mathsf{U_q}}
\newcommand{\WW}{\mathsf{W}}
\newcommand{\VV}{\mathsf{V}}
\newcommand{\ZZ}{\mathbb{Z}}
\newcommand{\CC}{\mathbb{C}}

\newcommand{\kk}{\mathbb{k}}

\newcommand{\TL}{\mathsf{TL}}
\newcommand{\wt}{\widetilde}
\newcommand{\End}{\mathsf{End}}
\newcommand{\Hom}{\mathsf{Hom}}
\newcommand{\ot}{\otimes}
\newcommand{\half}{\frac{1}{2}}

\newcommand{\modd}{\operatorname{\mathsf{mod}}}

\newcommand{\id}{\operatorname{\mathsf{id}}}
\newcommand{\gammarrow}{\buildrel { \mathbf{\gamma}} \over  \longrightarrow}
\newcommand{\twoarrow}{\buildrel {\mathbf{2}} \over  \longrightarrow}
\newcommand{\onearrow}{\buildrel {\mathbf{1}} \over  \longrightarrow}
\newcommand{\bc}{(bc)^{\frac{n-1}{2}}}

\title[Tensor Representations for the Drinfeld Double of the Taft Algebra]{Tensor Representations for the Drinfeld Double \\
 of the Taft Algebra}
 
\author[Benkart]{Georgia Benkart} 
\address[Benkart]{Department of Mathematics, University of Wisconsin-Madison, Madison, WI 53706, USA}

\author[Biswal]{Rekha Biswal} 
\address[Biswal]{Max Planck Institute for Mathematics, 53111 Bonn, Germany}  
\email{rekhabiswal27@gmail.com} 

\author[Kirkman]{Ellen Kirkman}
\address[Kirkman]{Department of Mathematics, Wake Forest University, Winston-Salem, NC 27109, USA} 
\email{kirkman@wfu.edu} 

\author[Nguyen]{Van C.~Nguyen}
\address[Nguyen]{Department of Mathematics, United States Naval Academy, Annapolis, MD 21402, USA}
\email{vnguyen@usna.edu} 

\author[Zhu]{Jieru Zhu}
\address[Zhu]{Hausdorff Research Institute for Mathematics, 53115 Bonn, Germany}
\email{jieruzhu699@gmail.com} 
  
\date{\today}
\allowdisplaybreaks

\makeatletter
\@namedef{subjclassname@2020}{%
  \textup{2020} Mathematics Subject Classification}
\makeatother

\subjclass[2020]{Primary 16T05. Secondary 16T25, 20C08}
\keywords{quasitriangular Hopf algebra, ribbon element, R-matrix, Temperley-Lieb algebra, Bratteli diagram, 
centralizer algebra}

\begin{document}

\begin{abstract}
{Over an algebraically closed field $\mathbb k$ of characteristic zero, the Drinfeld double $\Df_n$ of the Taft algebra that is defined using a primitive $n$th root of unity $q \in \mathbb k$ for $n \geq 2$ is a quasitriangular Hopf algebra. Kauffman and Radford have shown that $\Df_n$ has a ribbon element if and only if $n$ is odd, and the ribbon element is unique; however there has been no explicit description of this element. In this work, we determine the ribbon element of $\Df_n$ explicitly. For any $n\ge  2$, we use the R-matrix of $\Df_n$ to construct an action of the Temperley-Lieb algebra $\mathsf{TL}_k(\xi)$ with $\xi = -(q^{\half}+q^{-\half})$ on the $k$-fold tensor power $\VV^{\otimes k}$ of any two-dimensional simple $\Df_n$-module $\VV$.  This action is known to be faithful for arbitrary $k \ge 1$. We show that $\mathsf{TL}_k(\xi)$ is isomorphic to the centralizer algebra $\End_{\Df_n}(\VV^{\otimes k})$ for $1 \le k\le 2n-2$.}
\end{abstract}
\maketitle

\section{Introduction}  

Quasitriangular Hopf algebras are widely studied in representation theory, knot theory, tensor categories, and quantum physics \cite{EG,KauffR,LR}. Well-known examples include quantum groups associated to various Lie algebras and superalgebras. Motivated by the study of quantum groups, Drinfeld \cite{Dr} introduced the notion of a  quantum double of any Hopf algebra, and the resulting
quasitriangular Hopf algebra is now referred to as the Drinfeld double. 

The Drinfeld double $\Df_n$ of the Taft Hopf algebra $\Af_n$  at a primitive $n$th root of unity $q$ is an example of a finite-dimensional nonsemisimple Hopf algebra, and  
it admits more complex behavior than semisimple counterparts.
Representations of $\Df_n$ in characteristic zero have been studied in depth in  \cite{Chen99,Chen2.5,Chen2.75,Chen2,Chen,Chen3}. In particular, $\Df_n$ has simple modules $\VV(\ell,r)$, for $1\leq \ell\leq n$ and $r \in \ZZ_n = \ZZ/n\ZZ$,
and the fusion rules for decomposing tensor products of simple and projective modules are known.    
These fusion rules were used in \cite{BBKNZ} by the authors of this article to construct
various types of McKay matrices for $\Df_n$, to determine their eigenvalues and eigenvectors, and
to relate them to characters of $\Df_n$-modules.   Chebyshev polynomials of the second, third and fourth kinds
play an essential role in describing these eigenvectors and eigenvalues, and in expressing the characters of the simple
modules $\VV(\ell,r)$ when  they are evaluated on the grouplike elements of $\Df_n$.

The algebra $\Df_n$ is a quasitriangular Hopf algebra.  This entails the existence of a distinguished element of $\Df_n \ot \Df_n$, the so-called R-matrix, which induces a $\Df_n$-module endomorphism on tensor products.
The R-matrix can be used to define  an element $u$ with special properties (see \eqref{eq:udef}). 
In particular, $u$ is central and acts as a scalar on each simple module.  Some quasitriangular Hopf algebras admit a ribbon element $\upsilon$, which can be regarded as a square root of $u$. 
The ribbon element satisfies axioms compatible with the quasitriangular structure (see Section~\ref{S3.1}).  It has been used extensively to construct invariants of framed links in three-dimensional space and invariants of 3-manifolds  (see for example, \cite{H,Li1, Li2,Li3,RT1,RT2,RW}).

Our focus is on the quasitriangular structure of $\Df_n$ and on tensor modules of the form $\VV(2,r)^{\otimes k}$, where $\VV(2,r)$ is any of the two-dimensional  simple $\Df_n$-modules.
Kauffman and Radford \cite[Prop. 7]{KauffR} proved that $\Df_n$ has a ribbon element if and only if $n$ is odd, and that
the ribbon element of $\Df_n$ is unique. However, there has been no explicit description of this element.  In this work, we determine the unique ribbon element in $\Df_n$ and  show that  the following theorem holds.  In the statement,  $b$ and $c$ are two of the four generators of
$\Df_n$, and the grouplike elements of $\Df_n$ are exactly the elements $b^i c^k$ for $i,k \in \ZZ_n$.
The modules $\VV(\ell,r)$, $1 \le \ell \le n$, $r \in \ZZ_n$, are all the simple $\Df_n$-modules. \\

\begin{theorem}~{\rm(Theorem \ref{thm:ribbonelt})}  For $n$ odd, $n \ge 3$, and $q$ a primitive $n$th root of unity,
the unique ribbon element of the Drinfeld double $\Df_n$ of the Taft algebra $\Af_n$ is $\upsilon=ub^{\frac{n-1}{2}}c^{\frac{n-1}{2}} = u\bc$.  Moreover, 
$\upsilon$ acts  by the scalar $q^{r(r+\ell-1)+\frac{1}{2}(n-1)(\ell-1)}$ on the simple $\Df_n$-module $\VV(\ell,r)$ for
all $1 \le \ell \le n$ and $r \in \ZZ_n$.
\end{theorem}  
\vspace{-.05truein}

For any quasitriangular Hopf algebra $\Hf$  and any finite-dimensional $\Hf$-module $W$, the  R-matrix $\mathcal{R}$ of $\Hf$ can be used to construct elements of
$\End_{\Hf}(W^{\ot k})$  that satisfy axioms (QT1)-(QT3) in Section~\ref{S3.1}, where  (QT3)
is the well-known Yang-Baxter relation. As a result, tensor modules  for $\Hf$ admit an action of the braid group, which sometimes factors through an action of the Iwahori-Hecke algebra.
Leduc and Ram \cite{LR} describe a general recipe for constructing such an action on the $\Hf$-module $W^{\otimes k}$. Each generator $\sfs_i$, $1 \le i \le k-1$,  of the Iwahori-Hecke algebra acts by applying a scalar multiple of $\mathcal{R}$ to positions $i$ and $i+1$ of $W^{\otimes k}$ followed by switching those two tensor factors. The action of the Iwahori-Hecke algebra on $W^{\otimes k}$ is not faithful in general, and there is a kernel.  In Section~\ref{S4.1}, we describe the action of the Iwahori-Hecke algebra $\Hf_k(q)$ on $\VV^{\otimes k}$ for any two-dimensional simple $\Df_n$-module $\VV = \VV(2,r)$, $r \in \ZZ_n$.

The Temperley-Lieb algebra  is a quotient of the Iwahori-Hecke algebra,  and it has generators $\sft_i$, $1 \le i \le k-1$, that satisfy relations (R1')-(R4') in Section~\ref{S4.1}.   It was introduced in \cite{TL} to study the partition function of the Potts model of interacting spins in statistical mechanics and later shown to have  applications in the study of von Neumann algebras and subfactors \cite{JS}, tensor categories \cite{EG}, canonical bases of the quantum group $\Uq(\mathfrak{sl}_2)$ for
$\mathsf{q}$ generic \cite{FK}, $\mathfrak{sl}_2$-tensor invariants and webs \cite{CKM}, and countless other topics in mathematics and physics \cite{CES,K,KL,KR89}. 
The Temperley-Lieb algebra $\mathsf{TL}_k(\xi)$ that arises in our work is a quotient of the Iwahori-Hecke algebra $\Hf_k(q)$
and depends on the parameter $\xi = -(q^{\half}+q^{-\half})$. The main result of Section~\ref{SS4} is the following. \\

\begin{theorem}~{\rm (Theorem \ref{thm:TLalghom})}  Assume $q$ is a primitive $n$th root of unity for any  $n \ge 2$.
Let  $\VV= \VV(2,r)$, for  $r \in \ZZ_n$, be any two-dimensional simple $\Df_n$-module,  and set $\xi = -(q^{\half}+q^{-\half})$.  There is an injective algebra homomorphism
$\mathsf{TL}_k(\xi) \rightarrow \End_{\Df_n}(\VV^{\ot k})$ for $k \ge 2$  given
by $\sft_i\mapsto  q^{\half}(\lambda_r^{-1}\R_i- \id)$ for $1 \le i \le k-1$, where $\lambda_r = q^{-r(r+1)}$, and
$\R_i$ is the $\Df_n$-module homomorphism obtained by applying the {\rm R}-matrix of $\Df_n$ to tensor slots $i$ and $i+1$ of  $\VV^{\ot k}$ and then
interchanging those two factors.
\end{theorem}  

When $3 \leq \ell \leq n$ and $\VV$ is the $\Df_n$-module $\VV= \VV(\ell,r)$,  the action of the braid group  on $\VV^{\ot k}$ satisfies additional relations beyond the braid group relations,  and this setting is considered in Section \ref{dim3}.  When $n$ is odd, $n \geq 3$, and $\ell$ satisfies $2 \ell \leq n+1$, the ribbon element is used to compute the eigenvalues of $\R_i^2$ on $\VV(\ell,r)^{\otimes k}$ in Proposition \ref{obviousrel}.  In Proposition \ref{dim3andlarbitrary} we provide further relations on the action of $\R_i$ when $n$ is any integer $\geq 5$ and $\ell = 3$. For arbitrary $n \geq 2$ and $2 \ell \leq n+1$ we compute two eigenvalues of $\R$ on $\VV(\ell,r)^{\otimes 2}$ in Proposition \ref{toptwo} and suggest a formula for all eigenvalues of $\R_i$ on $\VV(\ell,r)^{\otimes k}$ in Conjecture \ref{conjecture}.

In Section~\ref{SS5}, we compare the Bratteli diagram for tensoring with a simple two-dimensional $\Df_n$-module $\VV$ with the Bratteli diagram for tensoring with the simple module $\CC^2$ for the quantum group $\Uq(\mathfrak{sl}_2)$ 
at a generic value of $\mathsf{q} \in \CC$.    In the $\Uq(\mathfrak{sl}_2)$-case, nodes in the $k$th row of the Bratteli diagram can be labeled by partitions
$\beta$ of $k$ with at most two parts. This corresponds to the fact that the finite-dimensional simple $\Uq(\mathfrak{sl}_2)$-modules have highest weights
that can be labeled by such partitions.   The comparison enables us to prove the following result giving the dimension of the centralizer algebra 
$\End_{\Df_n}(\VV^{\ot k})$ for any two-dimensional simple $\Df_n$-module $\VV$.     In the statement, we suppose $\Sf_1,\Sf_2,\dots, \Sf_{n^2}$ is a listing of the simple $\Df_n$-modules;
$\Pf_1, \Pf_2, \dots, \Pf_{n^2}$ are respectively their projective covers;   $\Ir_k$ is the set of all $i \in \{1,2,\dots,n^2\}$
such that $\Sf_i$ or $\Pf_i$ or both occur in $\VV^{\ot k}$;  $s_i$ (resp. $p_i$) is the multiplicity of
$\Sf_i$ (resp. $\Pf_i$) in $\VV^{\ot k}$ for $i \in \Ir_k$; and $\Pf_i, \Pf_i'$ are two projective covers with the same composition factors but arranged differently. \\

\begin{theorem}~{\rm (Theorem \ref{thm:dimCent})} \label{THM1.3} For any  $n\ge 2$ and
 any two-dimensional simple $\Df_n$-module $\VV$, the dimension of the centralizer algebra $\End_{\Df_n}\left(\VV^{\ot k}\right)$ for
$k \ge 1$ is 
$$\dimm_{\kk} \End_{\Df_n}\left(\VV^{\ot k}\right)
= \sum_{i \in \Ir_k} p_i^2 + \sum_{i \in \Ir_k} (s_i+p_i)^2 +  2 \sum_{i \in \Ir_k} p_ip_i'.$$
\end{theorem}

In this expression,  $s_i, p_i, p_i'$ represent the numbers of paths of length $k$ from $\VV$
 at level 1 in the Bratteli diagram to the summands $\Sf_i, \Pf_i, \Pf_i'$, respectively,  at level $k$ in the Bratteli diagram that is determined by $\VV$.
 

Using Theorem \ref{THM1.3} we prove our final main result.

\begin{theorem}{\rm(Theorem~\ref{doublecentralizer})} For any $n \ge 2$ and any $r \in \mathbb{Z}_n$, 
the algebra homomorphism $\pi: \mathsf{TL}_k(\xi)\to \End_{\Df_n}(\VV(2, r)^{\otimes k})$ is an isomorphism when $1\leq k\leq 2n-2$.
\end{theorem}

The module $\VV^{\otimes k}$, $\VV =\VV(2, r)$,  is completely reducible when $ 1\leq  k\leq n-1$, resulting in a relatively simple argument to prove this theorem 
for such values of $k$ by comparing the dimensions of $\mathsf{TL}_k(\xi)$  and $\End_{\Df_n}(\VV^{\otimes k})$. In particular, both dimensions can be regarded as sums of squares of the number of paths in a certain Bratteli diagram $\Gamma$, which is essentially a truncated Pascal's triangle (see  Lemma~\ref{observations}).  

When $n\leq k\leq 2n-2$, $\VV^{\otimes k}$ is no longer completely reducible, and it decomposes into a direct sum of simple and indecomposable projective modules that yields a modified Bratteli graph $\Gamma_n$  (displayed in Section~\ref{bratt:gamma5} for $n=5$ (Figure 1) and compared to the Bratteli diagram  $\Gamma$ of partitions with at most 2 parts  (Figure 2)). Nevertheless, the path counts in $\Gamma_n$ can be identified with those in the graph $\Gamma$ (as in Proposition~\ref{identifiedpathcount}), although the dimension argument is more complex. Once the dimension match for $k \le 2n-2$ is obtained, the isomorphism result follows immediately from the injectivity statement. When $k>2n-2$, examples 
in Table \ref{eq:Comp} (see Example \ref{Brat5}) for $n=5$ show that the dimension of $\End_{\Df_n}(\VV^{\otimes k})$ is larger than that of $\mathsf{TL}_k(\xi)$, and therefore the map fails to be an isomorphism beyond $k=2n-2$.   

This result is analogous to the well-known double-centralizer property \cite{CP,GW} between the quantum group $\Uq(\mathfrak{sl}_2)$ and the Temperley-Lieb algebras on tensor powers of 
the natural $\Uq(\mathfrak{sl}_2)$-module $\CC^2$  when $\mathsf{q}\in \CC$ is generic.
The centralizer results for $\Df_n$ are similar to those for the small quantum group $\mathsf{u}_q(\mathfrak{sl}_2)$ at $q$ a root of unity
(which is a quotient of $\Df_n$) acting on its unique simple two-dimensional self-dual module.

\vspace{-.25cm}
\section{Preliminaries} \label{S2.1} 

\emph{Throughout this work,  $n$ is an integer $\ge 2$,  $\mathbb{k}$ is an algebraically closed field of characteristic zero,  and $q$ is a primitive $n$th root of unity in $\mathbb{k}$.  All tensor products are over $\kk$, and we adopt Sweedler's notation for the coproduct $\Delta$ applied to an element $x$ of a Hopf algebra,}  
$$\Delta(x) = \sum_{x} x_{(1)} \ot x_{(2)}.$$ 

\noindent
{\bf Modules for the Drinfeld double of a Taft algebra.} 
The Drinfeld double $\Df_n$  has
a presentation as  the Hopf algebra over $\mathbb{k}$ with generators
$a,b,c,d$ that satisfy the following relations:
\begin{align} 
\begin{split} ba = q ab, \qquad  db = qbd,  & \qquad bc = cb, \qquad
ca = q ac,  \qquad  dc = qcd,  
 \qquad   da-qad = 1-bc, \\
 & \hspace{.2cm}  a^n = 0 = d^n,  \quad \quad b^n = 1 = c^n.  \end{split}\end{align}
It is an algebra of dimension $n^4$, and the elements $a^i b^j c^k d^\ell$,  $0 \le i,j,k,\ell \le n-1$, 
 determine a basis for $\Df_n$.
The coproduct, counit, and antipode of $\Df_n$ are given by
\begin{align}
\begin{split}\label{eq:hopfstructure}  \Delta(a) = a \otimes b + 1 \otimes a, & \quad  \Delta(d) = d \otimes c + 1 \otimes d,  \\
\Delta(b) = b \otimes b, & \quad  \Delta(c) = c \otimes c, \\
\varepsilon(a) = 0 = \varepsilon(d), & \quad  \varepsilon(b) = 1 = \varepsilon(c),\\
S(a) = -ab^{-1}, \  \ S(b) = b^{-1}, & \quad
S(c) = c^{-1}, \  \ S(d) = -dc^{-1}.  
\end{split}\end{align} 
The Taft algebra $\Af_n$ is the Hopf subalgebra generated by $a$ and $b$,
and Drinfeld's process of doubling $\Af_n$ to get $\Df_n$ was shown in \cite{Chen99} to yield
the presentation above.    It follows from \eqref{eq:hopfstructure} and the
fact that $\Delta$ is an algebra homomorphism
that the elements $g = b^ic^k$ for $0 \le i,k \le n-1$ are grouplike, that is, 
$\Delta(g) = g \ot g$, $\varepsilon(g) = 1 \in \mathbb{k}$, and $S(g) = g^{-1} = g^{n-1}$. 

The simple $\Df_n$-modules $\VV(\ell,s)$  are indexed by a pair $(\ell,s)$ where
$\ell \in \{1,2,\dots,n\}$ and $s \in \mathbb{Z}_n =  \mathbb{Z}/n\mathbb{Z}$  (the integers modulo $n$).   Then  $\VV(\ell,s)$ is a $\mathbb k$-vector space of dimension $\ell$ with basis
$v_1,v_2, \dots, v_\ell$  and with $\Df_n$-action given  by
\begin{align}
\label{eq:action} 
\begin{split} a. v_j &= v_{j+1}, \  1 \leq j < \ell, \hspace{2.2cm} a.v_\ell = 0, \\
b.v_j &= q^{s+j-1}v_j,  \hspace{3.3cm}  c.v_j = q^{j-(s+\ell)}v_j, \ \ 1 \le j \le \ell, \\
d.v_j &= \alpha_{j-1}(\ell) v_{j-1},  \ 1 < j \le \ell,  \hspace{0.9cm}  d.v_1 = 0,\quad \text{where}\end{split}  \end{align}
\begin{equation}\alpha_i(\ell) = \frac{ \left(q^i - 1\right)\left(1-q^{i-\ell}\right)}{q-1} \qquad \text{for}\; 1 \le i \le n-1.
\end{equation}

From \cite{Chen2.5, Chen}, we know the following:  
\begin{enumerate} 
\item $\VV(1,0)$ is the trivial $\Df_n$-module with action given by the counit $\varepsilon$.
\item $\VV(\ell, s) \otimes \VV(1,r)  \cong \VV(\ell, r+s)$.    
\item $\VV(\ell,s) \otimes \VV(\ell',r)$ is completely reducible if and only if $\ell+\ell' \le n+1$. In this case,  if $m = \mathsf{min}(\ell,\ell')$, 
then  \vspace{-.65cm}

\begin{equation}\label{eq:tensdecomp}\VV(\ell,s) \otimes \VV(\ell',r) \cong \bigoplus_{j=1}^m \VV(\ell+\ell' + 1-2j, r+s + j-1).\end{equation}
\end{enumerate}

Let $\Pf(\ell,s)$ be the projective cover of $\VV(\ell,s)$ for $1 \le \ell < n$ and $s \in \ZZ_n$. Chen \cite{Chen2.75}  has shown that any indecomposable projective left $\Df_n$-module is isomorphic to one of the 
modules $\Pf(\ell,s)$ for $1\le \ell < n$ or to $\VV(n,s)$ for some $s \in \mathbb{Z}_n$, where the indecomposable module $\Pf(\ell,s)$  
has the following structure:  

There is a chain of submodules $\Pf(\ell,s)  \supset  \mathsf{soc}^2(\Pf(\ell,s)
\supset  \mathsf{soc}(\Pf(\ell,s)) \supset (0)$ such that 
\begin{enumerate}
\item  $\mathsf{soc}(\Pf(\ell,s))$ is the socle of $\Pf(\ell,s)$ (the sum of all the simple submodules), and  $\mathsf{soc}(\Pf(\ell,s))\cong \VV(\ell,s)$;
\item $\mathsf{soc}^2(\Pf(\ell,s))/\mathsf{soc}(\Pf(\ell,s)) \cong  \VV(n-\ell, s+\ell) \oplus \VV(n-\ell, s+\ell)$;
\item $\Pf(\ell,s) /\mathsf{soc}^2(\Pf(\ell,s) \cong \VV(\ell,s)$.
\end{enumerate} 
The projective indecomposable modules $\Pf(\ell,s)$ and $\Pf(n-\ell, s+\ell)$  have the same composition factors, but arranged differently.
The dimension of $\Pf(\ell,s)$ is $2n$ for $1 \le \ell <n$.  The  
modules $\VV(n,s)$  for $s\in \mathbb{Z}_n$  are the only $\Df_n$-modules  that are both simple and projective.

It follows from \eqref{eq:tensdecomp} and  
\cite[Prop.~3.1, Thms.~3.3 and 3.5]{Chen}  that  for the two-dimensional simple module $\VV = \VV(2,r)$, 
\begin{enumerate}\label{eq:tens}
\item $\VV(1,s) \ot \VV  \cong \VV(2,r+s)$;
\item  $\VV(\ell,s) \ot \VV  \cong  \VV(\ell+1,r+s) \oplus \VV(\ell-1,r+1+s)$\; for $2 \le \ell < n$;   
\item $ \VV(n,s) \ot \VV   \cong \Pf(n-1, r+1+s)$;   
\item $ \Pf(1,s) \ot \VV  \cong  \Pf(2,r+s) \oplus 2 \VV(n,r+1+s)$;
\item $\Pf(\ell,s) \ot \VV  \cong  \Pf(\ell+1,r+s) \oplus \Pf(\ell-1,r+1+s)$\; for $2 \le \ell < n-1$;   
\item $\Pf(n-1,s) \ot \VV  \cong  \Pf(n-2,r+1+s) \oplus 2 \VV(n,r+s)$.   
\end{enumerate}

In Section~\ref{SS5}, we will use these relations to determine the decomposition of $\VV^{\ot k}$ into simple and projective summands and 
to  relate the corresponding Bratteli diagram to that obtained from tensor powers of the natural two-dimensional module
for the quantum group $\Uq(\mathfrak{sl}_2)$, where $\mathsf{q}$ is generic.

\section{The ribbon structure on $\Df_n$}

The algebra $\Df_n$ belongs to the class of  \emph{quasitriangular} Hopf algebras.  The distinguishing feature of a quasitriangular
Hopf algebra is the existence of an R-matrix, which induces a module homomorphism on tensor products.   The R-matrix can be used to
construct central elements in the algebra, and some quasitriangular Hopf algebras admit a special type of central element termed a \emph{ribbon element},
which plays an essential role in constructing knot and link invariants.   Kauffman and Radford \cite[Prop. 7]{KauffR} have shown that $\Df_n$ has a unique ribbon element
whenever $n$ is odd.  However, the exact expression for this ribbon element has not been known.    In this section, we review needed results
on quasitriangular Hopf algebras and use them to determine an explicit expression for the unique ribbon element of $\Df_n$. When the field is algebraically closed, the ribbon element acts as a
scalar on the simple $\Df_n$-modules because it is central, and we also determine that scalar for each simple $\Df_n$-module in Theorem \ref{thm:ribbonelt}.  

\subsection{Background on quasitriangular Hopf algebras}\label{S3.1}

Throughout Section~\ref{S3.1}, we assume $\Hf$ is a finite-dimensional Hopf algebra over $\mathbb k$ with antipode $S$, counit $\varepsilon$, and coproduct $\Delta(x) = \sum_{x} x_{(1)} \ot x_{(2)}$ for $x \in \Hf$, and $\Hf^*$ is the $\mathbb{k}$-dual Hopf algebra.    

\medskip
\noindent
{\bf The quasitriangular property.} (See, for example, \cite[Section 10.1]{Mon} or \cite[Section 4.2]{CP}.)  A Hopf algebra $\Hf$ is \emph{quasitriangular} if there is an invertible element $\mathcal{R} \in \Hf \ot \Hf$ such that 
\begin{enumerate}
\item  $\mathcal{R} \Delta(x) \mathcal{R}^{-1} = \Delta^{\mathsf{op}}(x)$ for all $x \in \Hf$, where $\Delta^{\mathsf{op}}(x)$ has
 the tensor factors in $\Delta(x)$ interchanged, and
\item $(\Delta \ot \mathsf{id})(\mathcal{R}) = \mathcal{R}_{13}\mathcal{R}_{23}$,
\item $(\mathsf{id} \ot \Delta)(\mathcal{R}) = \mathcal{R}_{13}\mathcal{R}_{12}$,
\end{enumerate}
where if $\mathcal{R}= \sum_i x_i \ot y_i$, then  
$\mathcal{R}_{12} =  \sum_i x_i \ot y_i \ot 1$, \;  $\mathcal{R}_{13} =  \sum_i x_i \ot 1\ot y_i$, and
$\mathcal{R}_{23} =  \sum_i 1 \ot x_i \ot y_i$.  \; Let  $\mathcal{R}^{\mathsf{op}}= \sum_i y_i \ot x_i $.   

\medskip
\noindent
{\bf The element $u$.}  We assume $\mathcal{R}= \sum_i x_i \ot y_i$
 as above  and use the antipode $S$ to define 
\begin{equation}\label{eq:udef}  u = \sum_{i} S(y_i) x_i \in \Hf. \end{equation}

\vspace{-.2cm}
Then the following hold\\
 \centerline{$u x u^{-1} = S^2(x)$  for all $x \in \Hf$ \;\; and \;\;  $\Delta(u) = \left(\mathcal{R}^{\mathsf{op}} \mathcal{R}\right)^{-1}(u \ot u)$.}

\medskip
\noindent
{\bf A ribbon Hopf algebra.}  
 A quasitriangular Hopf algebra $\Hf$ is a \emph{ribbon} Hopf algebra if there is an invertible element $\upsilon$
(the \emph{ribbon element}) in the center of $\Hf$ such that
\begin{equation}\label{eq:ribdef} \upsilon^2 = u S(u), \quad S(\upsilon) = \upsilon, \quad \varepsilon(\upsilon) = 1, \quad   \Delta(\upsilon) = \left(\mathcal{R}^{\mathsf{op}} \mathcal{R}\right)^{-1} (\upsilon \ot \upsilon),\end{equation}
 where $u$ is as in \eqref{eq:udef}.  Then $\upsilon^{-1}u$ is grouplike:   $\Delta(\upsilon^{-1}u) = \upsilon^{-1}u \ot \upsilon^{-1}u$.
 
\medskip
\noindent
{\bf The tensor power centralizer algebra.}  Let $\VV$ be  a module over the quasitriangular Hopf algebra $\Hf$,  and assume 
$\mathsf{R} \in \End_{\Hf}(\VV^{\ot 2})$
 gives the action of $\mathcal R$ on $\VV^{\ot 2}$.     Suppose $\sigma: \VV^{\ot 2} \rightarrow \VV^{\ot 2}$ is the interchange map $\sigma(w \ot x) = x \ot w$, and set $\R= \sigma \mathsf{R} \in \End_{\Hf}(\VV^{\ot 2})$.  Assume in $\End_{\Hf}(\VV^{\ot k})$ that
\begin{equation}
\label{eq:Ri} 
\R_i := \mathsf{id}_\VV \ot \cdots \ot  \mathsf{id}_\VV \ot \R \ot  \mathsf{id}_\VV \ot \cdots \ot \mathsf{id}_\VV,
\end{equation}
for $1 \le i \le k-1$, where $\R$ occupies tensor slots $i$ and $i+1$.   Then the following hold (see, for example,  \cite[Prop. 2.18]{LR}):
\begin{itemize} 
\item{\rm(QT1)} $\R_i$ belongs to the centralizer algebra $\End_{\Hf}(\VV^{\ot k})$ of transformations on $\VV^{\ot k}$ commuting with the $\Hf$-action.
\item{\rm(QT2)} $\R_i \R_j = \R_j \R_i$,  \qquad for $| i-j| > 1$. 
\item{\rm(QT3)} $\R_i \R_{i+1}\R_i = \R_{i+1} \R_i \R_{i+1}$,  \qquad for $1 \le i \le k-2$. 
\end{itemize}
This tells us that the subalgebra of  $\End_{\Hf}(\VV^{\ot k})$ generated by the $\R_i$ is a homomorphic image of the
group algebra of the braid group on $k-1$ strands. 
     
\medskip
\noindent
{\bf Action of the ribbon element.}  Suppose  $\{\UU_\omega\}_{\omega \in \Omega}$ are the simple $\Hf$-modules for the ribbon Hopf algebra
$\Hf$.   Then since the field is algebraically closed,  the ribbon element $\upsilon$ acts as a scalar, 
$\upsilon_\omega$  on $\UU_\omega$ by Schur's lemma.   
\begin{itemize}
\item When $\UU_\mu \ot \UU_\nu$ is a completely reducible $\Hf$--module,  then $\mathcal{R}^{\mathsf{op}} \mathcal{R}$
acts on a simple summand $\UU_\omega$ of $\UU_\mu \ot \UU_\nu$ by the scalar
\begin{equation}\label{eq:Req2} \frac{\upsilon_\mu \,\upsilon_\nu}{\upsilon_\omega}.\end{equation}
\end{itemize} 
More details can be found in \cite{LR}.  

\medskip
\noindent
{\bf Integrals and quasiribbon elements.} To determine the ribbon element of $\Df_n$ in the next section,  we will 
use several well-known facts about integrals and quasiribbon elements for a finite-dimensional Hopf algebra $\Hf$. 
For $h \in \Hf$ and $\alpha$  in the dual space $\Hf^*$,  we define
$$\langle \alpha, h \rangle := \alpha(h) \in \mathbb{k}.$$
 The following results on integrals can be found, e.g., in \cite[Secs.~12.1.1, 12.1.2]{L} or \cite[Chap.~2]{Mon}:  
\begin{itemize}
\item  The right integrals and left integrals of $\Hf$ are respectively  
\begin{equation*}\qquad\quad \textstyle{\int_{\Hf}^r} = \{ \Lambda \in \Hf \mid \Lambda h = \langle \varepsilon,h\rangle \Lambda \; \text{for all}  \; h \in \Hf\} \; \text{and}\;
\textstyle{\int_{\Hf}^\ell} = \{ \Lambda' \in \Hf \mid h\Lambda' = \langle \varepsilon,h\rangle \Lambda' \;  \text{for all} \;  h \in \Hf\}.
\end{equation*}
\item These spaces are one-dimensional (see \cite[Thm.~10.9(b)]{L}) and are related by the antipode:
$$S(\textstyle{\int_{\Hf}^r}) = \textstyle{\int_{\Hf}^\ell}, \qquad S(\textstyle{\int_{\Hf}^\ell}) = \textstyle{\int_{\Hf}^r}.$$
 When $\textstyle{\int_{\Hf}^\ell} = \textstyle{\int_{\Hf}^r}$,  then $\Hf$ is said to be \emph{unimodular}.  
\item  Fix $\Lambda \ne 0$  in $\textstyle{\int_{\Hf}^r}$.  For every $h \in \Hf$, $h \Lambda \in \textstyle{\int_{\Hf}^r}$, hence $h \Lambda$ is a scalar multiple of
$\Lambda$.  This implies that there is an $\tilde \alpha \in \Hf^*$, such that $h \Lambda = \tilde\alpha(h)  \Lambda$, for all $h \in  \Hf$.   The element $\tilde \alpha$ is a grouplike element of $\Hf^*$, referred to as the \emph{distinguished grouplike element of $\Hf^*$.}  The condition that $\Hf$ is unimodular is equivalent to $\tilde \alpha = \varepsilon$. 
\end{itemize}

Analogously, the dual algebra $\Hf^*$ has a  \emph{right integral},  which we denote $\lambda$, and for any $h^* \in \mathsf{H}^*$,
\begin{equation*}
\lambda h^* =  h^*(1_\Hf) \lambda.
\end{equation*}
Corresponding to a nonzero right integral $\lambda \in \mathsf{H}^*$, there is a  \emph{distinguished grouplike element} $\tilde g \in \mathsf{H}$ such that for any $h^*\in \mathsf{H}^*$, 
\begin{equation*}
h^* \lambda =  h^*(\tilde g) \lambda.
\end{equation*}

 As above, assume the R-matrix is
$\mathcal R = \sum_i x_i \ot y_i$, and define
\begin{equation}\label{eq:ha} 
g_{\tilde \alpha} =\sum_i x_i\,\tilde\alpha(y_i), \quad \text{and} \quad 
 h_{\tilde \alpha} =g_{\tilde \alpha}\,\tilde g^{-1},   
\end{equation}
where $\tilde \alpha$ is the distinguished grouplike element of $\Hf^*$,  and  $\tilde g$ is the distinguished grouplike element of $\mathsf{H}$.
 
A \emph{quasiribbon} element of the Hopf algebra $\mathsf{H}$ is an element satisfying all the ribbon conditions in \eqref{eq:ribdef} except for the requirement that it be central. 
Our approach to finding an explicit formula for the ribbon element of $\Df_n$ is to use the following results from \cite{KauffR} on quasiribbon elements. 

\begin{theorem}{\rm \cite[Thm.~1]{KauffR}}\label{ribboncomputation1}  Suppose $h_{\tilde \alpha}'$ is any element of $\Hf$ such 
that $(h_{\tilde \alpha}')^2 = h_{\tilde \alpha}$,  i.e.  $h_{\tilde \alpha}'$ is any square root of  the element $h_{\tilde \alpha}$ in \eqref{eq:ha}.
Then
$\upsilon=u h_{\tilde \alpha}'$ is a quasiribbon element,  where $u$ is as in \eqref{eq:udef}.
\end{theorem}
\begin{corollary}{\rm \cite[Cor.~2]{KauffR}}\label{ribboncomputation2}
When $\mathsf{H}$ has odd dimension, its quasiribbon element is unique.
\end{corollary}
   
\begin{remark}\label{rem:Dn-ribbon} Since the Drinfeld double $\Df_n$ of the Taft algebra $\Af_n$ has a unique ribbon element when $n$ is odd by  \cite[Prop.~7]{KauffR}, 
it is the only possible candidate for the quasiribbon element.  Moreover,   by \cite[Thm.~4]{Radford} or \cite[Prop.~3.4]{Chen},  $\Df_n$ is unimodular,
so that $\tilde \alpha = \varepsilon$. 
Thus for  $\Df_n$, the ribbon element $\upsilon=uh_{\tilde \alpha}'$, where $h_{\tilde \alpha}'$ is a square root of $h_{\tilde \alpha} = h_\varepsilon = g_\varepsilon \tilde g^{-1}$.
We will use this fact in the next section to compute $\upsilon$ explicitly and to determine its action on simple $\Df_n$-modules.  
\end{remark}

\subsection{Computation of the ribbon element in $\Df_n$}\label{S3.2} 

Throughout this section $n$ is an \emph{odd} integer $n \ge3$.
We combine the description of the ribbon element of $\Df_n$ in Remark \ref{rem:Dn-ribbon} with the next result
due to Radford to obtain an explicit expression for the ribbon element of $\Df_n$.       
We identify $\Df_n$ with $\Af_n^*\otimes \Af_n$, where $\Af_n$ is the Taft algebra, and $\Af_n^*$ is its dual,  and we assume that $\alpha_0$ and $g_0$ are the distinguished grouplike elements in $\Af_n^*$ and $\Af_n$, respectively.   Under these identifications, we have  

\begin{proposition}{\rm \cite[Cor.~7]{Radford}}\label{gplike}
The distinguished grouplike element in $\Df_n$ is given by $\alpha_0 \otimes g_0$.\end{proposition}

To understand this element and its relation to the ribbon element of $\Df_n$ more precisely,  we need a better grasp of the isomorphism identifying $\Df_n$ and
$\Af_n^* \ot \Af_n$, which is detailed in the following results of Chen.  (In comparing the statements below with the results of \cite{Chen99}, it is helpful to note
that our $\Af_n^*$  is $A_n(q)$, and our $\Af_n$ is $H_n(q)$ in the notation of \cite{Chen99}.)
The expressions for the coproduct in \eqref{eq:coproductab} and \eqref{eq:coproductcd} below require the quantum binomial coefficient
$${m \brack i}=\frac{[m]!}{[i]!\,[m-i]!}\, ,$$ 
for $m,i \in \ZZ_{\ge 0}$, with $m \ge i$.  Here $[m]!$ is the quantum factorial $[m]!=[m][m-1]\cdots[1]$,  where $[m]$
is the quantum integer $[m]=1+q+\cdots+q^{m-1}$ for $m \ge 1$, 
and $[0] = [0]! = 1$.
  
\begin{itemize}
\item [{\rm (i)}]  {\rm\cite[Thm.~3.3]{Chen99}}\label{Dnasdouble}
\emph{As a vector space $\Df_n=\Af_n^*\otimes \Af_n$. Explicitly,
\begin{align*}
a=1\otimes A, \hspace{.2 in} b=1\otimes B, \hspace{.2 in} c=C\otimes 1, \hspace{.2 in} d=D\otimes 1,
\end{align*}
where $A$ and $B$ are generators of $\Af_n$, $C$ and $D$ are generators of $\Af_n^*$.}
\item [{\rm (ii)}]  \emph{The product in $\Af_n^*\otimes \Af_n$ is given by
\begin{align}
(x\otimes X)(y\otimes Y)=\sum \tau(y_{(1)},X_{(1)}) xy_{(2)}\otimes X_{(2)}Y \tau^{-1}(y_{(3)},X_{(3)}). \label{bowtie}
\end{align}
In this expression, $\tau: \Af_n^*\otimes \Af_n \to \mathbb{k}$ is the skew bilinear form of \cite[Lemma~3.2]{Chen99} which satisfies axioms (SP.1) - (SP.4) of \cite[Sec. 1]{Chen99}.
In particular, $\tau(1,1)=\tau^{-1}(1,1)=1$.}
\item [{\rm (iii)}] {\rm\cite[Lemma 2.7]{Chen99}}  \emph{The following relations hold for the coproduct on $\Df_n$.
\begin{align}
\Delta(a^mb^w)&= \sum_{i=0}^m {m \brack i} a^ib^w \otimes a^{m-i}b^{w+i},  \label{eq:coproductab}\\
\Delta(c^md^w)&= \sum_{i=0}^w {w \brack i} c^md^i \otimes c^{m+i}d^{w-i}. \label{eq:coproductcd}
\end{align}}
\end{itemize}

Recall that multiplication $f\ast g$ in the dual algebra is given by 
\begin{align}\label{eq:dualprod}(f\ast g)(x)= \langle f\ast g, x \rangle = \sum_x  f(x_{(1)})\,g(x_{(2)})
\end{align}
for $f,g \in \Af_n^*$, $x \in \Af_n$.  Because of the identifications in (i) above, we will not distinguish between $a$ and $1\otimes A$ as a generator for $\Af_n$, and similarly for the other generators. Technically speaking, the symbols in the next lemma should be understood as their respective capitalized versions.
For a monomial basis element $a^ib^j \in \Af_n$, denote its dual basis element in $\Af_n^*$ by $(a^ib^j)^*$, and define $(c^id^j)^*$ similarly.

\begin{lemma}\label{grouplikesinAanddual} 
\begin{itemize} 
\item[{\rm (a)}]  $(a^{n-1}b)^*$ is a right integral in $\Af_n^*$, and the distinguished grouplike element of $\Af_n$  is $b$.  
\item[{\rm (b)}] $(cd^{n-1})^*$ is a right integral in $(\Af_n^*)^*$, and the distinguished grouplike element of $\Af_n^*$ is $c$.
\end{itemize}
\end{lemma}

\begin{proof} We use the following characterization of a right integral of $\Hf^*$  given in \cite[(12.2)]{L} with $\Hf = \Af_n$:
\begin{align} \lambda \in \textstyle{\int_{\Af_n^*}^r}  &\iff   \lambda(x) \,1_{\Af_n} =  \sum_x   
\lambda(x_{(1)})\,x_{(2)} \qquad \text{for all} \; \,  x \in \Af_n.  \label{eq:rinteq}  
\end{align}
{\rm (a)} For $\lambda = (a^{n-1}b)^*$ and $x = a^m b^w$, we have
\begin{align*} \big\langle (a^{n-1}b)^*, a^m b^w\big\rangle\, 1_{\Af_n} = \delta_{n-1,m}\, \delta_{1,w} \,1_{\Af_n} = \begin{cases} 1_{\Af_n} & \text{if} \;
m = n-1\; \text{and}\; w = 1 \, (\modd n),\\
0 & \; \text{otherwise}. \end{cases}
\end{align*}
By \eqref{eq:coproductab}, the expression on the right-hand side of \eqref{eq:rinteq} is
\begin{align*}  \sum_{i=0}^m {m \brack i}  \big \langle (a^{n-1}b)^*,a^ib^w \big \rangle \otimes a^{m-i}b^{w+i}.
\end{align*}
This is zero unless $i=n-1=m$ and $w = 1 \, (\modd n)$,  in which case $a^{m-i}b^{w+i} = a^0b^0 = 1_{\Af_n}$. 
Thus,  $(a^{n-1}b)^*$ is a right integral of $\Af_n^*$  by \eqref{eq:rinteq}.

To determine the distinguished grouplike element associated to $(a^{n-1}b)^*$, we use \eqref{eq:coproductab} and \eqref{eq:dualprod} to compute  
\begin{align*}
 \big\langle (a^sb^t)^* * (a^{n-1}b)^*, a^mb^w\big\rangle
=& \sum_{i=0}^m{m \brack i} \big \langle (a^sb^t)^*,a^ib^w\big \rangle \,\,\big \langle(a^{n-1}b)^*, a^{m-i}b^{w+i}\big\rangle \\
=&{m \brack s} \big\langle (a^sb^t)^*, a^sb^w\big\rangle\,\, \big \langle (a^{n-1}b)^*, a^{m-s}b^{w+s}\big\rangle \\
=& {m \brack s}\, \delta_{t,w}\,\delta_{n-1,m-s}\delta'_{1,w+s},
\end{align*}
where in the Kronecker delta $\delta'_{1,w+s}$ the term $w+s$ should be interpreted mod $n$.
This expression is zero, unless $w =1-s \, (\modd n)$ and  $m=n-1+s$. Therefore,  $(a^sb^t)^* * (a^{n-1}b)^*=0$ if $s\geq 1$, and when $s=0$, 
\begin{align*}
(b^t)^* * (a^{n-1}b)^*=\delta_{t,1}\,(a^{n-1}b)^*.
\end{align*}
Consequently, for $0 \le s \le n-1$, and $t \in \ZZ_n$, 
\begin{equation*}
(a^sb^t)^* * (a^{n-1}b)^* =  \langle(a^sb^t)^*, b \rangle \,(a^{n-1}b)^* = \begin{cases} 0  & \; \text{if}\;\, (s,t) \ne (0,1), \\
1 & \; \text{if}\;\, (s,t) = (0,1). \end{cases}
\end{equation*}
This shows that $b$ is the distinguished grouplike element in $\Af_n$.  

Part (b) follows by a similar argument using \eqref{eq:coproductcd}.
\end{proof}

Let  $\alpha_0$ and $g_0$  be  the distinguished grouplike elements defined prior to Proposition~\ref{gplike}. The previous lemma implies that $\alpha_0=C$ and $g_0=B$, under the identifications made in (i) of Section~\ref{S3.2}. Therefore $C\otimes B$ is the distinguished grouplike element in $\Df_n$ by Proposition~\ref{gplike}.

 Recall that the R-matrix $\mathcal{R}=\sum_{i} x_i \ot y_i$ gives rise to the element $u = \sum_i S(y_i)x_i$
 in Section~\ref{S3.1}.  We will use the explicit expression 
 for the R-matrix of the quasitriangular Hopf algebra $\Df_n$ in 
\cite[Thm.~3.6]{Chen99}:
\begin{align}\label{eq:Rmat}
\mathcal{R}=\frac{1}{n} \sum_{m,s,t=0}^{n-1} \frac{q^{-tm}}{[s]!}a^sb^t \otimes c^md^s.
\end{align} 
 
\begin{theorem}\label{thm:ribbonelt} Assume that $n$ is an odd integer $n \geq 3.$\begin{itemize}   \item [{\rm (a)}]
The unique ribbon element in $\Df_n$ is
\begin{align*}
\upsilon=u\,\bc, \quad \text{where}
\end{align*}
\begin{align}
u 
\; &= \;  \frac{1}{n} \sum_{m,s,t=0}^{n-1} \frac{q^{\frac{s(s-1)}{2} -t(m+s)}}{[s]!}\, (-1)^s \,d^s c^{-s-m}b^ta^s. \label{eq:uexplicit}
\end{align} 
\item [{\rm (b)}] The ribbon element  $\upsilon$ of $\Df_n$ acts on any simple $\Df_n$-module  $\VV(\ell,r)$ by the scalar
\begin{align*}
q^{r(r+\ell-1)+\frac{1}{2}(n-1)(\ell-1)} \qquad \text{for \ $1 \le \ell \le n$ and $r \in \ZZ_n$}.
\end{align*}
\end{itemize}
\end{theorem}

\begin{proof}
\noindent (a) We adopt the previous conventions and set $g_{\tilde \alpha}=g_{\varepsilon}$,  (which holds as  $\Df_n$ is unimodular). 
Now combining \eqref{eq:Rmat}  and \eqref{eq:ha}, we have
\begin{equation} g_{\varepsilon} =\frac{1}{n} \sum_{m,s,t=0}^{n-1} \frac{q^{-tm}}{[s]!}a^sb^t \,\varepsilon(c^md^s).
\end{equation}
Since $\varepsilon(d)=0$ and $\varepsilon(c)=1$,  and $\varepsilon$ is an algebra homomorphism, only the terms  with $s=0$ survive, and therefore 
\begin{align*}
g_{\varepsilon}&= \frac{1}{n}\sum_{m,t=0}^{n-1} q^{-tm} b^t=\frac{1}{n}\sum_{t=0}^{n-1}\left (\sum_{m=0}^{n-1} q^{-tm}\right) b^t.
\end{align*}
Observe that 
\begin{align*}
\sum_{m=0}^{n-1} q^{-tm}=\frac{1-(q^{-t})^n}{1-q^{-t}}=0
\end{align*}
unless $t=0$, in which case $\sum_{m=0}^{n-1}q^{-tm}=n$. Therefore $g_{\varepsilon}=1_{\Df_n}$.

By the discussion prior to this theorem, the distinguished grouplike element  in $\Df_n\simeq \Af_n^{*}\otimes \Af_n$ is $\tilde g=C\otimes B$. We aim to show that this is in fact $cb$. The expression for the product in $\Df_n\simeq \Af_n^{*}\otimes \Af_n$ in (\ref{bowtie}) when applied to $cb$ reduces to a single term:
\begin{align*}
cb=(C\otimes 1)(1\otimes B)= \tau(1,1)C\otimes B \tau^{-1}(1,1)=C\otimes B=\tilde g.
\end{align*}
This is mainly due to the fact that $\Delta(1)=1\otimes 1$,  and the coproduct factors of $\Delta^2(1)$ are  $1_{(1)}=1_{(2)}=1_{(3)}=1$.

By (\ref{eq:ha}), $h_\varepsilon= g_{\varepsilon}\tilde{g}^{-1}=(bc)^{-1} = (bc)^{n-1}$. When $n$ is odd, the square root $h_\varepsilon'$ of $h_\varepsilon$ is unique,  because $h_\varepsilon$, and therefore $h_\varepsilon'$,  has odd order. Thus,   
\begin{equation}\label{eq:ribbon}
h_\varepsilon' = (bc)^{\frac{n-1}{2}} \quad \text{and} \quad \upsilon=uh_\varepsilon' = u\,\bc
\end{equation} 
by Theorem~\ref{thm:ribbonelt},  as asserted in part (a).   

It remains to show that $u$ has the expression in \eqref{eq:uexplicit}.  Recall that
$u=\sum_i S(y_i)x_i$, where $\mathcal{R}=\sum_i x_i\otimes y_i = \frac{1}{n} \sum_{m,s,t=0}^{n-1} \frac{q^{-tm}}{[s]!}\,a^sb^t \otimes c^m d^s.$  Thus,  
\begin{align*}
u \;  &= \;  \frac{1}{n} \sum_{m,s,t=0}^{n-1} \frac{q^{-tm}}{[s]!}S(c^m d^s) a^sb^t = 
\frac{1}{n} \sum_{m,s,t=0}^{n-1} \frac{q^{-tm}}{[s]!}(-dc^{-1})^s c^{-m} a^sb^t \notag \\
\; &= \;  \frac{1}{n} \sum_{m,s,t=0}^{n-1} \frac{q^{\frac{s(s-1)}{2}-tm}}{[s]!} (-1)^s\,d^s c^{-s-m}a^sb^t = 
\frac{1}{n} \sum_{m,s,t=0}^{n-1} \frac{q^{\frac{s(s-1)}{2}-t(m+s)}}{[s]!} (-1)^s\,d^s c^{-s-m}b^ta^s.
\end{align*}  
As a result, (a) holds.  

(b) Since $\upsilon$ is central and the field is assumed to be algebraically closed, $\upsilon$ acts on any simple module $\VV(\ell,r)$ by a scalar.   
It suffices to compute the action of $\upsilon$ on any vector of $\VV(\ell,r)$, in particular, on the last basis element  $v_{\ell}$, and for this we have
\begin{align*}
bc.v_{\ell}=q^{r+\ell-1}q^{-r-\ell+\ell}v_{\ell}=q^{\ell-1}v_{\ell}. 
\end{align*}
Note that in the expression  \eqref{eq:uexplicit} for $u$,  all the summands act on $v_{\ell}$ as zero except for the terms with $s=0$. The scalar action of $u$ on $v_{\ell}$ reduces to the action of 
$$\overline{u}=\frac{1}{n} \sum_{m,t=0}^{n-1} q ^{-mt } c^{-m}b^t$$
on $v_\ell$, and the scalar determined by $\overline u$ is  
\begin{align*}
\frac{1}{n}\sum_{m,t=0}^{n-1} q^{-mt} q^{-m(-r)}q^{(r+\ell-1)t} 
 \; &=\; \frac{1}{n} \sum_{t=0}^{n-1} q^{t(r+\ell-1)}\,\sum_{m=0}^{n-1}q^{m(r-t)}=q^{r(r+\ell-1)}.
\end{align*}
The last equality uses the fact that $\sum_{m=0}^{n-1}q^{m(r-t)}=0$ unless $t=r$, and when $t=r$ its value is equal to $n$.
 Part (b) follows from multiplying the scalars for the action of $u$ and $(bc)^{\frac{n-1}{2}}$ on $v_\ell$.  \end{proof}

 \section{Temperley-Lieb actions}\label{SS4} 
 
We continue to assume that $\mathbb k$ is an algebraically closed field of characteristic zero,  and  take $q \in \mathbb k$ to be a primitive $n$th root of unity for any $n \ge 2$.  For $k \geq 2$, we describe an action of the Iwahori-Hecke algebra $\mathsf{H}_k(q)$
on $\VV(2,r)^{\ot k}$ for $r \in \ZZ_n$, that comes from the R-matrix of $\Df_n$.    We show that this action factors 
through a Temperley-Lieb algebra.   Section 5 will continue our focus on tensor powers of $\VV(2,r)$, $r \in \ZZ_n$.    
 
\subsection{Action of the Iwahori-Hecke algebra}
\label{S4.1} 

Assume $k$ is an integer $\geq 2$. The group algebra of the braid group on $k-1$ strands has generators $\sfs_1$, $\sfs_2$, \dots, $\sfs_{k-1}$ over $\mathbb k$ subject to relations (R1) and (R2) below.  The  \emph{Iwahori-Hecke algebra} (of type A) is the associative algebra  $\mathsf{H}_k(q)$ over $\mathbb k$ that is the quotient of the group algebra of the braid group by the additional relation (R3) below:
\begin{align}
&\sfs_i \sfs_j=\sfs_j\sfs_i \quad \hspace{1.2 in} |i-j|>1, \tag{R1}\label{R1}\\
&\sfs_i \sfs_{i+1}\sfs_i=\sfs_{i+1}\sfs_i\sfs_{i+1}\quad  \hspace{.3 in} \; 1\leq i\leq k-2,  \tag{R2} \label{R2}\\
&(\sfs_i-1)(\sfs_i+q^{-1})=0\quad \hspace{.3 in} 1\leq i\leq k-1 \label{R3}. \tag{R3}
\end{align}
Let $\sft_i=q^{\frac{1}{2}}(\sfs_i-1)$, and fix the parameter $\xi=-(q^{\frac{1}{2}}+q^{-\frac{1}{2}})$. It is straightforward to check that 
$\mathsf{H}_k(q)$ has an alternative presentation with generators $\sft_1,\sft_2,\dots,\sft_{k-1}$ subject to relations
\begin{align}
&\sft_i\sft_j=\sft_j\sft_i \hspace{1.75 in} |i-j|>1, \tag{R1'}\label{R1'}\\
&\sft_i\sft_{i+1}\sft_i-\sft_{i}=\sft_{i+1}\sft_i\sft_{i+1}-t_{i+1}\quad \;\;1\leq i\leq k-2,  \tag{R2'} \label{R2'}\\
&\sft_i^2=\xi \sft_i \hspace{1.75 in} 1\leq i\leq k-1 \label{R3'}. \tag{R3'}
\end{align} 

The \emph{Temperley-Lieb algebra}  $\TL_k(\xi)$ with $\xi=-(q^{\frac{1}{2}}+q^{-\frac{1}{2}})$ is the quotient of $\mathsf{H}_k(q)$ by the further relations (compare to \cite{FK,Mor}):
\begin{align}
\sft_i\sft_{i + 1}\sft_i =\sft_i\tag{R4'}  \hspace{.3 in} 1\leq i\leq k-2. \label{R4}
\end{align}
The dimension of the Temperley-Lieb algebra $\TL_k(\xi)$ is the Catalan number $\mathcal{C}_k$ for any 
value of $\xi$.    Thus, $\dimm \TL_1(\xi) = \dimm \TL_0(\xi) = 1$, 
$\dimm \TL_2(\xi) = 2$,  $\dimm \TL_3(\xi) = 5$,  $\dimm \TL_4(\xi) = 14$, and so forth.

\medskip

\noindent
{\bf The action of the Iwahori-Hecke algebra on tensor powers.} Let $\VV$ be a module over a quasitriangular Hopf algebra $\Hf$. Recall that the mappings $\R_i := \mathsf{id} \ot \cdots \ot  \mathsf{id} \ot \R \ot  \mathsf{id} \ot \cdots \ot \mathsf{id}$ in \eqref{eq:Ri}, 
where $\R$ occupies tensor slots $i$ and $i+1$ for $1 \le i \le k-1$, 
satisfy the relations (QT1)-(QT3). Thus, there is a linear action of the braid group on the $k$-fold tensor power $\VV^{\otimes k}$ given by 
\[\sfs_i \mapsto \alpha\R_i, \quad \text{ for some } \alpha \in \mathbb k^{\times}. \] 
This is the standard action of the  braid group with $k-1$ strands on the tensor power $\VV^{\otimes k}$, and it is well-studied in the literature, see e.g., \cite{Jimbo, LR}. 

When $\VV = \VV(2,r)$, $r \in \ZZ_n$, is any two-dimensional $\Df_n$-module, we use the R-matrix of $\Df_n$ to explicitly describe an action of the Iwahori-Hecke algebra $\Hf_k(q)$ on $\VV^{\ot k}$ that induces an action of the Temperley-Lieb algebra $\TL_k(\xi)$ on $\VV^{\ot k}$. Note that this action does not depend on $r$ and the resulting representation coincides with the standard action on tensor powers described above.   

Recall that the R-matrix for $\Df_n$ is given by
\begin{align}
\label{eq:Rmat2}
\mathcal{R}=\frac{1}{n} \sum_{m,s,t=0}^{n-1} \frac{q^{-tm}}{[s]!}\,a^sb^t \otimes c^md^s
 =\frac{1}{n} \sum_{m,s,t=0}^{n-1} \frac{q^{-tm-ts}}{[s]!}\,b^t a^s\otimes c^md^s.
\end{align} 
We assume  $\mathsf{R} \in \End_{ \Df_n}(\VV^{\ot 2})$
gives the action of $\mathcal R$ on $\VV^{\ot 2}$, where the terms $a^s b^t$ act on the first tensor
factor of $\VV^{\ot 2}$ and the $c^m d^s$ on the second.
Let 
 $\sigma: \VV^{\ot 2} \rightarrow \VV^{\ot 2}$ be the map interchanging the factors,  $\sigma(w \ot x) = x \ot w$,  and set  $\R= \sigma \mathsf{R} \in 
\End_{ \Df_n}(\VV^{\ot 2})$.  

\medskip

\emph{For ease of notation, in what follows we will use $\id$ for the identity map; the space it is acting upon should be apparent from the context.}
\medskip

\begin{lemma}
\label{relationR3}  Let $\VV = \VV(2,r)$ for any $r \in \ZZ_n$. Then the transformation $\R$
on $\VV^{\otimes 2}$  satisfies 
\begin{equation}
\label{eq:lambr}
(\R-\lambda_r\id)(\R+\lambda_rq^{-1}\id)=0, \quad \text{where} \quad  \lambda_r =q^{-r(r+1)}.\end{equation}
The eigenspace of $\R$ corresponding to the eigenvalue $\lambda_r$ is spanned by $v_1\otimes v_1$,   $v_1\otimes v_2+q^r v_2\otimes v_1$, and $v_2\otimes v_2$. The eigenspace corresponding to $-\lambda_rq^{-1}$ is spanned by $v_1\otimes v_2-q^{r+1}v_2\otimes v_1$.
\end{lemma}

 \begin{proof}  By \eqref{eq:action},  the actions of the generators of
 $\Df_n$ relative to the basis elements $\{v_1, v_2\}$ of $\VV = \VV(2,r)$ are given by the following matrices:
 $$a \rightarrow \left(\begin{matrix} 0 & 0 \\ 1 & 0 \end{matrix}\right) \qquad
b \rightarrow \left(\begin{matrix} q^r & 0 \\ 0 & q^{r+1} \end{matrix}\right) \qquad
c \rightarrow \left(\begin{matrix} q^{-(r+1)} & 0 \\ 0 & q^{-r} \end{matrix}\right)
\qquad d \rightarrow \left(\begin{matrix} 0 & \alpha_1(2) \\ 0 & 0 \end{matrix}\right),$$
where $\alpha_1(2) = 1-q^{-1}$.
 
To compute the action of $\mathsf{R}$ on $\VV^{\ot 2}$,   note that $a^s$  and $d^s$ act as 0 on $\VV$ for $s \ge 2$.  
 Then  
 \begin{align*} \mathsf{R}(v_1 \ot v_1) &= \frac{1}{n} \sum_{ m,t =0}^{n-1} q^{-tm} \, b^t v_1 \ot c^m v_1  \\
 &= \frac{1}{n} \sum_{ m,t =0}^{n-1} q^{-tm+tr-m(r+1)} \, v_1 \ot v_1  
 = \frac{1}{n}\left(\sum_{m=0}^{n-1} q^{-m(r+1)} \bigg(\sum_{t = 0}^{n-1} q^{t(r-m)}\bigg)\right) v_1 \ot v_1 \\
 &= q^{-r(r+1)}v_1 \ot v_1, 
 \end{align*}
since the inner summation over $t$ is 0 unless $m=r$, in which case it is $n$.  
 Therefore,  
 \begin{equation} 
 \label{eq:R11} 
 \R(v_1 \ot v_1) \,= \,\sigma \mathsf{R} (v_1 \ot v_1) = q^{-r(r+1)}v_1 \ot v_1 = \lambda_r v_1 \ot v_1.
 \end{equation}
 Similarly,   
 \begin{align*} \mathsf{R}(v_2 \ot v_2) &= \frac{1}{n} \sum_{ m,t =0}^{n-1} q^{-tm} \, b^t v_2 \ot c^m v_2  
= \frac{1}{n} \sum_{ m,t =0}^{n-1} q^{-tm+t(r+1)-mr} \, v_2 \ot v_2 \\ 
& = \frac{1}{n}\left(\sum_{m=0}^{n-1} q^{-mr} \bigg(\sum_{t = 0}^{n-1} q^{t(r+1-m)}\bigg)\right) v_2 \ot v_2 
= q^{-r(r+1)}v_2 \ot v_2,  \; \; \text{so that}  \end{align*}
 \begin{equation} 
 \label{eq:R22} \R(v_2 \ot v_2) \,=\, \lambda_r v_2 \ot v_2.\end{equation}
 \noindent  Now
  \begin{align*} \mathsf{R}(v_1 \ot v_2) &= \frac{1}{n} \sum_{m,t =0}^{n-1} q^{-tm} \, b^t v_1 \ot c^m v_2 + 
   \frac{1}{n} \sum_{m,t =0}^{n-1} q^{-t(m+1)} \, b^t v_2  \ot  c^m \alpha_1(2) v_1 \\ 
&\hspace{-.8cm}= \left(\frac{1}{n} \sum_{m,t =0}^{n-1} q^{-tm+tr-mr}\right) \, v_1 \ot  v_2 +  
 \frac{\alpha_1(2)}{n} \left(\sum_{m,t =0}^{n-1} q^{-t(m+1)+t(r+1)-m(r+1)}\right) v_2 \ot  v_1 \\
 &\hspace{-.8cm} = \frac{1}{n}\left(\sum_{m=0}^{n-1} q^{-mr} \bigg(\sum_{t = 0}^{n-1} q^{t(r-m)}\bigg)\right) v_1 \ot v_2 +
 \frac{\alpha_1(2)}{n} \left(\sum_{m =0}^{n-1}q^{-m(r+1)}\bigg(\sum_{t = 0}^{n-1} q^{t(r-m)}\bigg)\right)v_2 \ot v_1
  \\  &\hspace{-.8cm} = q^{-r^2}v_1 \ot v_2 + (1-q^{-1})q^{-r(r+1)}v_2 \ot v_1.  \quad \text{Therefore,}
 \end{align*}
\begin{equation}\label{eq:R12} \R(v_1 \ot v_2) \,=\, (1-q^{-1})q^{-r(r+1)}v_1 \ot v_2 + q^{-r^2}v_2 \ot v_1.\end{equation}
 \noindent Finally,  
 \begin{align*} \mathsf{R}(v_2 \ot v_1) &= \frac{1}{n} \sum_{m,t =0}^{n-1} q^{-tm} \, b^t v_2 \ot c^m v_1 
 = \frac{1}{n} \sum_{m,t =0}^{n-1} q^{-tm+t(r+1)-m(r+1)} \, v_2 \ot v_1 \\
 & = \frac{1}{n}\left(\sum_{m=0}^{n-1} q^{-m(r+1)} \bigg(\sum_{t = 0}^{n-1} q^{t(r+1-m)}\bigg)\right) v_2 \ot v_1 = q^{-(r+1)^2} v_2 \ot v_1,
 \;\;\; \text{and as a result,}
 \end{align*}
\begin{equation}\label{eq:R21} \R(v_2 \ot v_1) \,=\,  q^{-(r+1)^2}v_1 \ot v_2.\end{equation}
On the subspace of $\VV^{\ot 2}$ spanned by $v_1 \ot v_2$ and $v_2 \ot v_1$, the transformation $\R$ has matrix
$$\left(\begin{matrix}  \lambda_r -\lambda_rq^{-1} & q^{-(r+1)^2} \\ q^{-r^2} & 0 \end{matrix}\right).$$
It is easily seen that this matrix has eigenvalues $\lambda_r = q^{-r(r+1)}$ and $-\lambda_rq^{-1} = -q^{-r(r+1)-1}$ 
with corresponding eigenvectors $v_1 \ot v_2 + q^{r}v_2 \ot v_1$,
$v_1 \ot v_2 - q^{r+1} v_2 \ot v_1$,  respectively.     
Thus, $\R$ has a three-dimensional eigenspace on $\VV^{\ot 2}$ with basis $v_1 \ot v_1, v_1 \ot v_2 + q^{r}v_2 \ot v_1$, $v_2 \ot v_2$
corresponding to the eigenvalue $\lambda_r$, and a one-dimensional eigenspace with basis element
 $v_1 \ot v_2 - q^{r+1}v_2 \ot v_1$ corresponding to the eigenvalue $-\lambda_rq^{-1}$.
\end{proof} 
  
Recall that the mappings $\R_i := \mathsf{id} \ot \cdots \ot  \mathsf{id} \ot \R \ot  \mathsf{id} \ot \cdots \ot \mathsf{id}$ in \eqref{eq:Ri}, 
where $\R$ occupies tensor slots $i$ and $i+1$ for $1 \le i \le k-1$, 
satisfy the relations (QT1)-(QT3). Hence, they belong to $\End_{\Df_n}(\VV(2,r)^{\ot k})$.  As we noted earlier, the subalgebra of  $\End_{\Df_n}(\VV(2,r)^{\ot k})$ generated by the $\R_i$ is a homomorphic image of the
group algebra of the braid group on $k-1$ strands, but in fact,  we can now 
 prove that the centralizer algebra of the $\Df_n$-action on $\VV^{\ot k}$ affords a representation of the Iwahori-Hecke algebra.

\begin{proposition}\label{prop:Heckeaction}
Let  $\VV = \VV(2,r)$, $r \in \ZZ_n$, and assume $k \ge 2$. There is an algebra homomorphism
$\mathsf{H}_k(q) \rightarrow \End_{\Df_n}(\VV^{\ot k})$ given by $\sfs_i \mapsto \lambda_r^{-1}\R_i$ 
for $1 \le i \le k-1$, where $\R_i$ is as in \eqref{eq:Ri} and $\lambda_r = q^{-r(r+1)}$.
\end{proposition}
\begin{proof}  We know from \eqref{eq:lambr} that $\R$ satisfies the relation
$(\R - \lambda_r\id)(\R+\lambda_r q^{-1}\id) = 0$ on $\VV^{\ot 2}$ when $\VV= \VV(2,r)$. Because $\R_i$ is just
the operator $\R$ applied to factors $i$ and $i+1$ in $\VV^{\ot k}$,  and  $\lambda_r \neq 0$,  
we have that the relation 
\begin{equation}\label{eq:Rilambda}\left(\lambda_r^{-1}\R_i  -\id\right)\left(\lambda_r^{-1} \R_i+q^{-1}\id\right) = 0\end{equation}
holds on $\VV^{\ot k}$.
Therefore, the map $\sfs_i \mapsto \lambda_r^{-1}\R_i$ for $1 \le i \le k-1$ extends to an algebra homomorphism
$\mathsf{H}_k(q) \rightarrow \End_{\Df_n}(\VV^{\ot k})$,  since the $\sfs_i$ and  $\lambda_r^{-1}\R_i$
satisfy the same relations (see (QT1)-(QT3) above).  Relation
 (R3) of the definition of the Iwahori-Hecke algebra is just \eqref{eq:Rilambda}.
 
\end{proof}

The standard action of the Iwahori-Hecke algebra $\mathsf{H}_k(q)$ on a $k$-fold tensor power $\VV^{\otimes k}$ factors through an action of the Temperley-Lieb algebra when dim$(\VV)=2$, see e.g., \cite[Remark 3]{Jimbo}. Using the action of $\mathsf{H}_k(q)$ on $\VV(2,r)^{\ot k}$ given by Proposition \ref{prop:Heckeaction}, we obtain the following result.

\begin{theorem}
\label{thm:TLalghom}  
Assume $q$ is a primitive $n$th root of unity for  $n\ge 2$. 
Let  $\VV = \VV(2,r)$ for any $r \in \ZZ_n$, and set  $\xi = -(q^{\half}+q^{-\half})$.  Then for $k \geq 2$,
\begin{itemize}
 \item[{\rm(a)}] There is an algebra homomorphism
$\pi:\mathsf{TL}_k(\xi) \rightarrow \End_{\Df_n}(\VV^{\ot k})$ given
by $\sft_i \mapsto  q^{\half}(\lambda_r^{-1}\R_i- \id)$ for $1 \le i \le k-1$, where $\lambda_r = q^{-r(r+1)}$,  and
 \item[{\rm(b)}] $\pi$ is an injection, that is, the action of $\mathsf{TL}_k(\xi)$ on $\VV^{\otimes k}$ is faithful.
 \end{itemize} 
 
 \end{theorem} 

\begin{proof} (a) Suppose that $\R_i$ is the twist of the R-matrix of $\Df_n$ applied
to the tensor factors $i$ and $i+1$ of $\VV^{\ot k}$ as before.   
Recall that by Proposition \ref{prop:Heckeaction}, the map $\sfs_i \mapsto \lambda_r^{-1}\R_i$, $1 \le i \le k-1$,  extends to an algebra homomorphism $\Hf_k(q) \rightarrow
\End_{\Df_n}(\VV^{\ot k})$.  Therefore, setting $\sft_i = q^{\half}(\sfs_i-\id)$, we know that the
relations (R1'), (R2') (R3') for the alternative presentation of $\Hf_k(q)$ hold for $\xi = -(q^{\half}+q^{-\half})$.  
Consequently, there is an algebra homomorphism $\Hf_k(q) \rightarrow
\End_{\Df_n}(\VV^{\ot k})$ given by $\sft_i \mapsto  q^{\half}(\lambda_r^{-1}\R_i- \id)$ for $1 \le i \le k-1$.   What remains to be
shown is that  $\sft_i \sft_{i+1}\sft_i - \sft_i  \mapsto 0$ so that  $\sft_i \sft_{i+1}\sft_i - \sft_i$ is in the kernel of this homomorphism for all $1 \le i \le k-1$,
and there is an induced algebra homomorphism $\pi: \mathsf{TL}_k(\xi) \rightarrow \End_{\Df_n}(\VV^{\ot k})$ by relation (\ref{R4}). 

 It suffices  to check that $\sft_1 \sft_{2}\sft_1 = \sft_1$ on $\VV^{\ot 3}$,  as $\sft_i \sft_{i+1}\sft_i -\sft_i$ acts as the identity map on all factors of $\VV^{\ot k}$, except  for the factors
in positions $i, i+1,$ and $i+2$, and the action is the same as for $\sft_1\sft_2 \sft_1 -\sft_1$   on $\VV^{\otimes 3}$.   By the proof of Lemma~\ref{relationR3}, on the basis $\{v_1 \ot v_1, v_1 \ot v_2, v_2 \ot v_1, v_2 \ot v_2\}$ for $\VV^{\ot 2}$, the transformation $\R$ has matrix
$$\left(\begin{matrix}  \lambda_r & 0 & 0 & 0 \\
0 & \lambda_r -\lambda_rq^{-1} & q^{-(r+1)^2} & 0 \\ 
0 & q^{-r^2} & 0 & 0 \\
0 & 0 & 0 & \lambda_r \end{matrix}\right).$$
 Since  $\sft_1 =   q^{\half}(\lambda_r^{-1}\R_1- \id)$ it follows that $\sft_1$ is zero on the vectors  $v_1 \ot v_1 \ot v_1, v_1 \ot v_1 \ot v_2, v_2 \ot v_2 \ot v_1,$ and $v_2 \ot v_2 \ot v_2$ so that   $\sft_1 \sft_{2}\sft_1 = \sft_1$  holds on the subspace spanned by these four vectors. 
Using the fact that $\sft_1\sft_2\sft_1-\sft_1$ is central in $\Hf_3(q)$, it can be shown that  $\sft_1 \sft_{2}\sft_1 = \sft_1$  on all of   $\VV^{\ot 3}$. For example
\[ v_2 \ot v_1 \ot v_2 =  q^{(r+1)^2}\R_2 (v_2 \ot v_2 \ot v_1) 
=  q^{(r+1)^2}\lambda_r (q^{-\half}\sft_2 +   \id)(v_2 \otimes v_2 \otimes v_1)
\]
so that
\begin{align*} (\sft_1 \sft_{2}\sft_1 - \sft_1)(v_2 \ot v_1 \ot v_2 ) &=  q^{(r+1)^2}\lambda_r (\sft_1 \sft_{2}\sft_1 - \sft_1) (q^{-\half}\sft_2 +   \id) ) (v_2 \ot v_2 \ot v_1)\\
&=  q^{(r+1)^2}\lambda_r(q^{-\half}\sft_2 +   \id) ) (\sft_1 \sft_{2}\sft_1 - \sft_1)  (v_2 \ot v_2 \ot v_1) = 0.\end{align*}
Then similarly using \begin{align*}v_1 \ot v_2 \ot v_2 &=   q^{(r+1)^2}\R_1 (v_2 \ot v_1 \ot v_2),\\
v_1 \ot v_2 \ot v_1 &=   q^{r^2}(\R_2- \lambda_r(1-q^{-1}) \id) (v_1 \ot v_1 \ot v_2),\\ 
v_2 \ot v_1 \ot v_1 &=    q^{r^2}(\R_1- \lambda_r(1-q^{-1}) \id)  (v_1 \ot v_2 \ot v_1),\end{align*} 
we  obtain that $\sft_1 \sft_{2}\sft_1 = \sft_1$  on all of   $\VV^{\ot 3}$.

(b) The action that comes from the standard Iwahori-Hecke action and passing to the Temperley-Lieb action is known to be faithful, see e.g., \cite[Theorem 2.4]{GW} and \cite[Main Theorem]{Martin}.
\end{proof}

\subsection{Action for $\VV(3,r)$ } \label{dim3}

In this section we give an application of the explicit formula in Theorem~\ref{thm:ribbonelt} for the ribbon element when $n$ is odd.  We show that it can be used to  compute the eigenvalues of $\R$ to obtain further relations on the action of the braid group on $\VV(\ell,r)^{\otimes k}$ for arbitrary integers $1 \leq \ell \leq n,~ r \in \ZZ_n$, and $n \geq 2$.  We illustrate these results in the case when $\ell=3$.  The condition $2\ell \leq n+1$ implies that $\VV(\ell, r)^{\ot 2}$ is completely reducible, see (\ref{eq:tensdecomp}).

Our first result holds for $n$ odd and  arbitrary $\ell$, with $2 \ell \leq n + 1$, and uses the action of the ribbon element.

\begin{proposition} 
\label{obviousrel}
Let  $n \geq 3$ be odd, $2\ell \leq n+1$, and $\VV=\VV(\ell,r)$, $r \in \ZZ_n$, be a simple $\Df_n$-module of dimension $\ell$.  The braid group action where the generators act via $\check{\mathsf{R}}_i$ on $\VV^{\otimes k}$ satisfies the further relations for $1\leq i\leq k-1$:
\begin{align*}
\prod_{\UU} (\R_i^2- c_{\UU}\id)=0,
\end{align*}
where the product is over all simple $\Df_n$-modules $\UU$ occurring in the tensor product $\VV^{\otimes 2}$.  For $\UU=\VV(a,b)$ the scalar $c_\UU$ is computed as the follows:
\begin{align*}
c_\UU=q^{2r^2+(2r-1)(\ell-1)-b(a+b-1)-\frac{1}{2}(n-1)(a-1)}.
\end{align*}
Furthermore, the eigenvalues of $\R_i$ on $\VV^{\otimes k}$ come from among the values $\{\pm \sqrt{c_\UU}\}$, where $\UU$ ranges over all irreducible summands of $\VV^{\otimes 2}$. 
\end{proposition}

\begin{proof} We use the eigenspace decomposition of $\VV^{\otimes 2}= \VV(\ell, r)^{\ot 2}$ under the action of $\R$. Let $W \subseteq \VV^{\otimes 2}$ be an eigenspace corresponding to eigenvalue $\alpha$. It is straightforward to check that $W$ is a $\Df_n$-module. Since $\VV^{\ot 2}$ is completely reducible, $W$ has a direct sum decomposition $W= \bigoplus_j W_j$ into irreducible $\Df_n$-summands $W_j$. For $1\leq i\leq k-1$, since $\R_i$ is just the operator $\R$ applied to factors $i$ and $i+1$ in $\VV^{\ot k}$, its eigenspace corresponding to eigenvalue $\alpha$ is $\VV^{\ot i-1} \ot \big(\bigoplus_j W_j \big) \ot \VV^{\ot k-i-1}$. That is, $\R$ and $\R_i$ share the same eigenvalues. 

Property (1) of the quasitriangular properties in Section~\ref{S3.1} is equivalent to $\check{\mathsf{R}}\Delta(x)=\Delta(x)\check{\mathsf{R}}$  for all  $x \in \Df_n$.  In other words,  the action of $\check{\mathsf{R}}$  on a simple summand  $\UU_{\omega}\subset \UU_{\mu}\otimes \UU_{\nu}$ should be a scalar,  because $\check{\mathsf{R}}\in \operatorname{End}_{\mathsf{D}_n}(\UU_{\omega})$,  and the latter space is one-dimensional by Schur's Lemma.  Since $\mathcal{R}^{\operatorname{op}}\mathcal{R}=\check{\mathsf{R}}^2$, this scalar can be computed using  (\ref{eq:Req2}) and Theorem \ref{thm:ribbonelt} (b). This gives the desired expression for $c_\UU$. As $\alpha^2 = c_{\UU}$ is an eigenvalue of $\R_i^2$, we see from the above discussion that an eigenvalue of $\R_i$ is a square root of $c_{\UU}$, determined up to a sign.
\end{proof}

The remaining computations give a refinement of the relation in Proposition~\ref{obviousrel} by finding eigenvalues for the action of $\R_i$ rather than their squares. In other words, we determine the sign of the square root of $c_{\UU}$.

The following result gives the full list of eigenvalues specifically when $\ell=3$ and $n$ is any integer, $n \geq 5$,  and hence we have an additional relation for the action of the braid group on  $\VV(3,r)^{\otimes k}$.

\begin{proposition}
\label{dim3andlarbitrary}
Let $n$ be an arbitrary integer $n \geq  5$.  The action of the braid group  on $\VV(3,r)^{\otimes k}$ factors through the further relations.
\begin{align*}
(\R_i-q^{-r^2-2r}\id)(\R_i+q^{-r^2-2r-2}\id)(\R_i-q^{-r^2-2r-3}\id)=0,  \hspace{.35 in} \text{ for }  1\leq i\leq k-1.
\end{align*}
\end{proposition}

\begin{proof} Since $\R_i$ is just the operator $\R$ acting on two tensor factors of $\VV(3,r)^{\otimes k}$, it suffices to prove the statement for the action of $\R$ on $\VV(3,r)^{\otimes 2}$. 

For a $\Df_n$-module $\VV(3,r)$ with basis $v_1, v_2,$ and $v_3$,
$$\VV(3,r)\otimes \VV(3,r)\simeq \VV(5,2r) \oplus \VV(3, 2r+1)\oplus \VV(1,2r+2).$$  On the three summands, $\R$ acts by $\pm q^{-r^2-2r}$,  $\pm q^{-r^2-2r-2}$, $\pm q^{-r^2-2r-3}$ respectively, where each of the eigenvalues of $\R$ has either the plus or minus sign.  By computing the action of $\R$ on $v_1\otimes v_1$  one obtains the first eigenvalue   $q^{-r^2-2r}$ (a more general computation will be given in the next proposition).   We now show the second eigenvalue takes the minus sign. We note that for $s \geq3$, $a^s$ and $d^s$ act  as $0$ on $\VV(3,r)$.
A straightforward  computation shows  that 
\begin{align*}
\mathsf{R}(v_1\otimes v_2) =q^{r(-r-1)}v_1\otimes v_2 +q^{r(-r-2)}(1-q^{-2})v_2\otimes v_1,  \hspace{.4 in} \mathsf{R}(v_2\otimes v_1) = q^{(-r-2)(r+1)}v_2\otimes v_1.
\end{align*}
Therefore,  on the subspace spanned by $v_1\otimes v_2$ and $v_2\otimes v_1$,  $\R = \sigma \mathsf{R}$ acts via the matrix 
$$\left(\begin{matrix}
(1-q^{-2})q^{r(-r-2)} & q^{(-r-2)(r+1)} \\ q^{(-r-1)r} & 0
\end{matrix}\right),$$ and  therefore it has an eigenvalue $-q^{(-r^2-2r-2)}$ with eigenvector $v_1\otimes v_2 -q^{r+2} v_2 \otimes v_1$.

For the remaining one-dimensional module,  we compute the action of $\R$ on the subspace spanned by $v_1\otimes v_3$,  $v_2\otimes v_2$,  $v_3\otimes v_1$.  Here the summation is over all $0\leq t,m\leq n-1$ unless otherwise specified. We have 
\begin{align*}
\mathsf{R} (v_1\otimes v_3) &= \frac{1}{n}\sum_{m,t} q^{-mt-2t}\frac{1}{[2]}b^ta^2v_1\otimes c^md^2 v_3+\frac{1}{n}\sum_{m,t}q^{-mt-t}b^tav_1\otimes c^mdv_3+\frac{1}{n}\sum_{m,t}q^{-mt}b^tv_1\otimes c^mv_3\\
&= \frac{\alpha_1(3)\alpha_2(3)}{[2]n}\sum_{m,t}q^{-mt-2t+(r+2)t+(-r-2)m}v_3\otimes v_1 +\frac{\alpha_2(3)}{n}\sum_{m,t}q^{-mt-t+(r+1)t+(-r-1)m}v_2\otimes v_2\\
&\qquad +\frac{1}{n}\sum_{m,t}q^{-mt+tr-rm}v_1\otimes v_3\\
&= \frac{\alpha_1(3)\alpha_2(3)}{[2]n}\sum_{t}(q^{tr}\sum_{m}q^{m(-t-r-2)})v_3\otimes v_1+\frac{\alpha_2(3)}{n}\sum_{t}(q^{tr}\sum_{m}q^{m(-r-1-t)})v_2\otimes v_2\\
& \qquad +\frac{1}{n}\sum_t (q^{tr}\sum_m q^{m(-t-r)})v_1\otimes v_3\\
&= \frac{\alpha_1(3)\alpha_2(3)}{[2]} q^{r(-r-2)}v_3\otimes v_1 +\alpha_2(3)q^{r(-r-1)}v_2\otimes v_2+q^{-r^2}v_1\otimes v_3.
\end{align*}
By similar calculations, one computes the action on $v_2\otimes v_2$ and $v_3\otimes v_1$,  and  obtains that the matrix of the action of $\R-q^{-r^2-2r-3}\id$ on the subspace spanned by $v_1\otimes v_3$,  $v_2\otimes v_2$,  $v_3\otimes v_1$ is
\begin{align*}
\left(\begin{matrix}
\frac{1}{[2]}\alpha_1(3)\alpha_2(3) q^{r(-r-2)}-q^{-r^2-2r-3} & q^{-(r+2)(r+1)}\alpha_1(3) & q^{-(r+2)^2}\\
q^{-r(r+1)}\alpha_2(3) & q^{-(r+1)^2}-q^{-r^2-2r-3} & 0 \\
q^{-r^2} & 0 & -q^{-r^2-2r-3}\\
\end{matrix}\right).
\end{align*}
This matrix has rank $2$, since $\mathbf{u}-q^{-r-1}\mathbf{v}+q^{-2r-1}\mathbf{w}=\mathbf{0}$, where the first, second, and third rows of the matrix (in that order) are $\mathbf{u},
\mathbf{v},$ and 
$\mathbf{w}$.  Therefore $q^{-r^2-2r-3}$ is an eigenvalue.  We have recovered all eigenvalues for the action of $\R$ on $\VV(3,r)^{\otimes 2}$.
\end{proof}

\begin{remark}
One may argue that on the given summands, Proposition \ref{dim3andlarbitrary} is a stronger result than Proposition~\ref{obviousrel} as it gives us all eigenvalues of $\R_i$ explicitly for \emph{any} $n \geq 5$, not just for $n$ odd.  However, in practice,  when Proposition \ref{obviousrel} applies, it gives the eigenvalues up to a sign, and it is much easier to verify that a scalar is an eigenvalue than to find it without prior knowledge.   
\end{remark}

Next we give a first attempt at computing eigenvalues of the action of $\R_i$ on $\VV(\ell,r)^{\ot  k}$ for any integer $\ell$ with $2 \ell \leq n+1$ by computing two eigenvalues of $\R$ on $\VV(\ell,r)^{\otimes 2}$.

\begin{proposition} 
\label{toptwo} 
Assume that  $n \geq 2$ and $2 \ell \leq n+1$. Then $q^{r(1-r-\ell)}$ and $-q^{-r^2+r-r\ell-\ell+1}$ are eigenvalues of $\R$ on $\VV(\ell,r)^{\otimes 2}$. 
\end{proposition}

\begin{proof}
We show that $v_1\otimes v_1$ is an eigenvector corresponding to eigenvalue $q^{r(1-r-\ell)}$: 
\begin{align*}
\mathsf{R} (v_1\otimes v_1) &=\frac{1}{n}\sum_{m,t=0}^{n-1}q^{-mt}q^{tr}q^{m(1-r-\ell)}v_1\otimes v_1\\
&= \frac{1}{n}\sum_{m=0}^{n-1}q^{m(1-r-\ell)}\sum_{t=0}^{n-1}q^{(r-m)t} v_1\otimes v_1 \\
&= q^{r(1-r-\ell)} v_1\otimes v_1.
\end{align*}
Similarly, one can compute that on the subspace spanned by $v_1\otimes v_2$ and $v_2\otimes v_1$,  $\R=\sigma \mathsf{R}$ acts via the following matrix
\begin{align*}
\left(\begin{matrix}
q^{r(1-r-\ell)}(1-q^{1-\ell}) & q^{(r+1)(1-r-\ell)}\\ q^{r(2-r-\ell)} & 0
\end{matrix}\right),
\end{align*}
which has eigenvalue $-q^{-r^2+r-r\ell-\ell+1}$.
\end{proof}

We give a conjecture for the action of $\R_i$ on the tensor power $\VV(\ell,r)^{\ot k}$ of a simple $\Df_n$-module.  

\begin{conjecture} 
\label{conjecture} 
For any simple $\Df_n$-module $\VV(\ell,r)$ with $n \geq 2$ and $2\ell \leq n+1$, the braid group action where the generators act via  $\check{\mathsf{R}}_i$ on $\VV(\ell,r)^{\otimes k}$ satisfies  the further relations for $1\leq i\leq k-1$:
\begin{align*}
\prod_{j=1}^{\ell} (\R_i- c_j  \id)=0.
\end{align*}
The product is over all simple $\Df_n$-modules $\VV(a_j,b_j)$ occurring in $\VV(\ell,r)^{\otimes 2}$, with $a_j=2\ell+1-2j$, \ $b_j=2r+j-1$, \ and $c_j=(-1)^{j+1}q^{ r^2+\frac{1}{2}(2r-1)(\ell-1)-\frac{1}{2}b_j(a_j+b_j-1)-\frac{1}{4}(n-1)(a_j-1)}$.
\end{conjecture}

\section{Bratteli diagrams and centralizer algebras}\label{SS5}

In this section we focus on the $k$-fold tensor power of any two-dimensional simple $\Df_n$-module $\VV=\VV(2,r)$ for any integer $n \geq 2$.  Our main result, Theorem~\ref{doublecentralizer}, provides an isomorphism between the centralizer algebra $ \End_{\Df_n}(\mathsf{V}^{\otimes k})$ and the Temperley-Lieb algebra $\mathsf{TL}_k(\xi)$ for $1 \leq k \leq 2n-2$.

 \subsection{Comparing Bratteli diagrams} \label{S5.1}

We establish a correspondence between the Bratteli diagram $\Gamma$ of partitions of at most two parts as in Section~\ref{S5.1}
and the Bratteli diagram $\Gamma_n$ which captures the decomposition of  
tensoring repeatedly with the simple $\Df_n$-module  $\VV = \VV(2,0)$. 
The vertices in $\Gamma_n$ lie in rows labeled by positive integers, where the vertices in Row $k$ correspond to the simple 
and projective direct summands of $\mathsf{V}(2,0)^{\otimes k}$. There are $s$ directed edges from $\WW$ to $\WW'$, if and only if $\WW'$ occurs as a summand  of $\WW\otimes \mathsf{V}(2,0)$ with multiplicity $s$. 

\begin{remark}Since $\VV(2,r)^{\ot k} \cong \VV(2,0)^{\ot k} \ot \VV(1,kr)$ for any  
$r \in \ZZ_n$, it suffices to consider the case $r =0$.  The Bratteli diagram for repeated tensoring with $\VV(2,r)$ can be obtained from one for $\VV(2,0)$ below simply by adding $kr$ to the second coordinate of each label in Row $k$ of $\Gamma_n$. \end{remark}
 
In Figure 1 we display Rows $k =1,2,\dots,11$ of the Bratteli graph $\Gamma_5$.
We use $(\ell,r)$ to signify the simple $\Df_n$-module $\mathsf{V}(\ell,r)$, and $\overline{(\ell,r)}$ for its projective cover $\mathsf{P}(\ell,r)$, where the second component $r$ should be read modulo $n=5$. 
The directed edges come from the decomposition formulas in Section~\ref{S2.1} for tensoring with $\VV = \VV(2,0)$.  A dotted arrow means the absence of an edge  (we include it as a ``fake'' edge for visual completeness), and a double arrow from $\WW$
to $\WW'$ means that two copies of $\WW'$ occur in $\WW \ot \VV(2,0)$.
The vertical dashed lines will be explained later on.  In Figure 2 we display the first 11 rows of  the graph $\Gamma$ of partitions with at most two parts.  
Thus on row $k$ of $\Gamma$, a partition $\beta \vdash k$ is represented  by a pair $\beta = \{\beta_1,\beta_2\}$ 
where $\beta_1 \ge \beta_2$, and $\beta_1+\beta_2 = k$. 
\ytableausetup{smalltableaux}
\begin{center}
\scalebox{0.6}{
\xymatrix{
\text{\Large Row}&&&& &&& &&& & &&\\
\text{\Large 1}&&&&\gamma=2\ar@{--}[dd]&&& & &\gamma=1 \ar@{--}[dd] & & & (2,0) \ar[dl]\ar[dr]&\\
\text{\Large 2}&&&&\ar@{--}[dd]&&& & & \ar@{--}[d] & &  (3,0)\ar[dl]\ar[dr]& & (1,1) \ar[dl]\\
\text{\Large 3}&&&&\ar@{--}[dd]&&& & & \ar@{--}[d]& (4,0)\ar[dl]\ar[dr]& &(2,1)\ar[dl]\ar[dr] &\\
\text{\Large 4}&&&&\ar@{--}[d] &&& & & (5,0) \ar@{--}[dd] \ar[dl]\ar@{-->}[dr] & & (3,1) \ar[dl]\ar[dr]& &(1,2)\ar[dl]\\
\text{\Large 5}&&&&\ar@{--}[d]&&& & \overline{(4,1)}  \ar[dl]\ar@{=>}[dr]& & (4,1)\ar[dl]\ar[dr]& & (2,2)\ar[dl]\ar[dr] &\\
\text{\Large 6}&&&&\ar@{--}[d]&&&\overline{(3,2)} \ar[dl] \ar[dr]  & & (5,1)\ar@{--}[dd] \ar[dl] \ar@{-->}[dr] & & (3,2) \ar[dl]\ar[dr] & & (1,3) \ar[dl]\\
\text{\Large 7}&&&&\ar@{--}[d]&&\overline{(2,3)} \ar[dl]\ar[dr] & & \overline{(4,2)} \ar[dl] \ar@{=>}[dr] & & (4,2) \ar[dl]\ar[dr] & & (2,3) \ar[dl]\ar[dr] &\\
\text{\Large 8}&&&& \ar@{--}[d] &\overline{(1,4)}\ar@{=>}[dl]  \ar[dr] && \overline{(3,3)} \ar[dl] \ar[dr]  && (5,2) \ar@{--}[dd] \ar[dl] \ar@{-->}[dr] && (3,3) \ar[dl] \ar[dr] && (1,4) \ar[dl]\\
\text{\Large 9}&&&&(5,5) \ar@{--}[dd] \ar[dl]\ar@{-->}[dr]&&\overline{(2,4)}  \ar[dl]\ar[dr]&&\overline{(4,3)} \ar[dl]\ar@{=>}[dr] &&(4,3) \ar[dl]\ar[dr]  && (2,4)\ar[dl]\ar[dr]& \\
\text{\Large 10}&&&\overline{(4,6)} \ar[dl]\ar@{=>}[dr] && \overline{(1,5)}\ar@{=>}[dl]  \ar[dr] &&\overline{(3,4)} \ar[dl]\ar[dr] && (5,3) \ar@{--}[d] \ar[dl] \ar@{-->}[dr]  && (3,4) \ar[dl]\ar[dr]&& (1,5) \ar[dl]\\
\text{\Large 11}&&\overline{(3,7)}&&(5,6)&&\overline{(2,5)} &&\overline{(4,4)} &\, & (4,4) &&(2,5)\\
&& && && && && && \\ }}
\end{center}

\vspace{-.6cm}
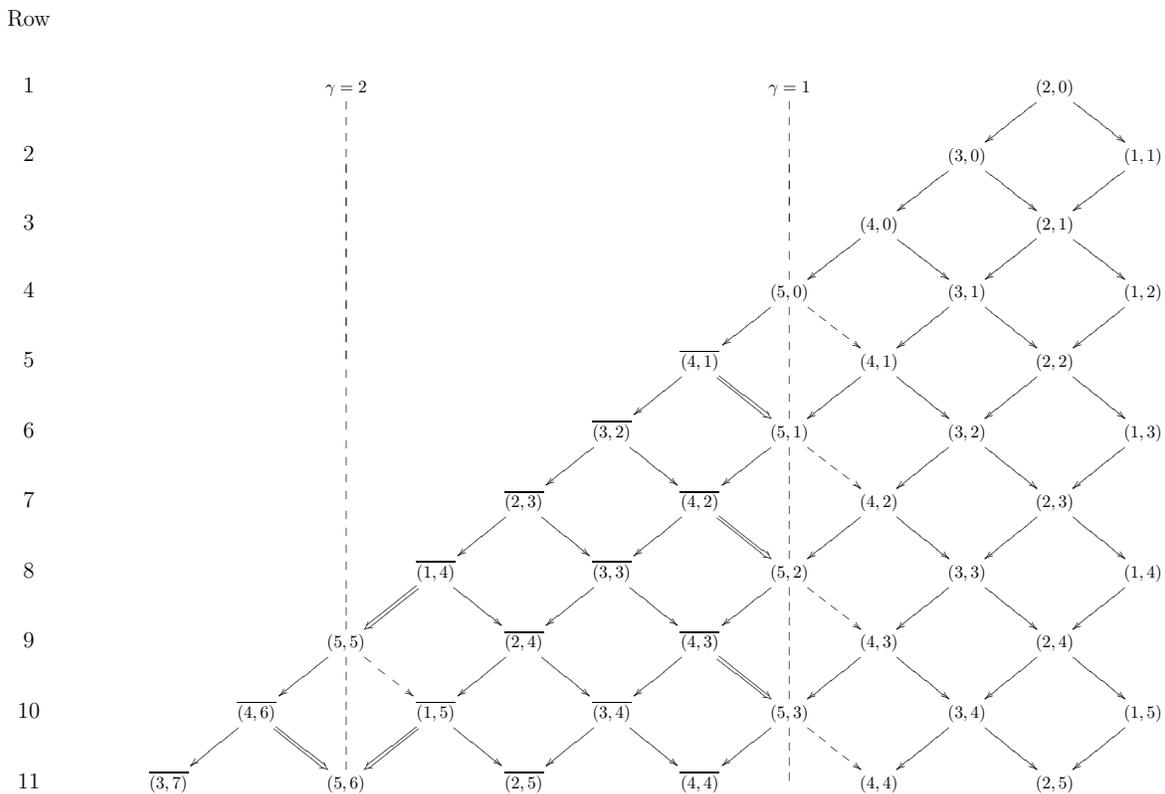
\begin{center}{\small   Figure 1: First 11 rows of the Bratteli diagram $\Gamma_5$}.\end{center} \label{bratt:gamma5}

\vspace{-.5cm}
\ytableausetup{smalltableaux}
\begin{center}
\scalebox{0.6}{
\xymatrix{
\text{\Large Row}&&&&&&& && & & &&\\
\text{\Large 1}&&&&\gamma=2\ar@{--}[dd]&&& & &\gamma=1 \ar@{--}[dd] & & & \{1\} \ar[dl]\ar[dr] &\\
\text{\Large 2}&&&&\ar@{--}[dd]&&& & & \ar@{--}[d] & & \{2\}\ar[dl]\ar[dr]& & \{1,1\}\ar[dl]\\
\text{\Large 3}&&&&\ar@{--}[dd]&&& & & \ar@{--}[d]&  \{3\}\ar[dl]\ar[dr]& &\{2,1\}\ar[dl]\ar[dr] &\\
\text{\Large 4}&&&&\ar@{--}[d] &&& & & \{4\}\ar[dl]\ar[dr] \ar@{--}[dd] \ar[dl]\ar@{-->}[dr]& &  \{3,1\} \ar[dl]\ar[dr]& &\{2,2\}\ar[dl]\\
\text{\Large 5}&&&&\ar@{--}[d]&&& &  \{5\}\ar[dl]\ar[dr]& & \{4,1\}\ar[dl]\ar[dr]& & \{3,2\}\ar[dl]\ar[dr] &\\
\text{\Large 6}&&&&\ar@{--}[d]&&& \{6\} \ar[dl] \ar[dr]   
& & \{5,1\}\ar[dl]\ar[dr] \ar@{--}[dd] & &  \{4,2\} \ar[dl]\ar[dr] & &  \{3,3\} \ar[dl]\\
\text{\Large 7}&&&&\ar@{--}[d]&& \{7\} \ar[dl]\ar[dr] & &  \{6,1\}\ar[dl]\ar[dr] & & \{5,2\} \ar[dl]\ar[dr] & & 
 \{4,3\}  \ar[dl]\ar[dr] &\\
\text{\Large 8}&&&& \ar@{--}[d] &\{8\}\ar[dl]\ar[dr] && \{7,1\} \ar[dl] \ar[dr]  && \{6,2\}\ar[dl]\ar[dr]\ar@{--}[dd]&& \{5,3\} \ar[dl] \ar[dr] && \{4,4\} \ar[dl]\\
\text{\Large 9}&&&&\{9\} \ar@{--}[dd]\ar[dl]\ar[dr] &&\{8,1\}  \ar[dl]\ar[dr]&&\{7,2\}\ar[dl]\ar[dr]  &&\{6.3\} \ar[dl]\ar[dr]  && \{5,4\}\ar[dl]\ar[dr] \\
\text{\Large 10}&&&\{10\} \ar[dl]\ar[dr]  && \{9,1\}\ar[dl]\ar[dr]&&\{8,2\} \ar[dl]\ar[dr] && \{7,3\} \ar@{--}[d] \ar[dl]\ar[dr]  && \{6,4\} \ar[dl]\ar[dr]&&\{5,5\} \ar[dl]\\
\text{\Large 11}&&\{11\} &&\{10,1\}&&\{9,2\} &&\{8,3\}&{}& \{7,4\} &&\{6,5\} \\
&&&&&&& && & & &&\\ }}
\end{center}

\vspace{-.5cm}
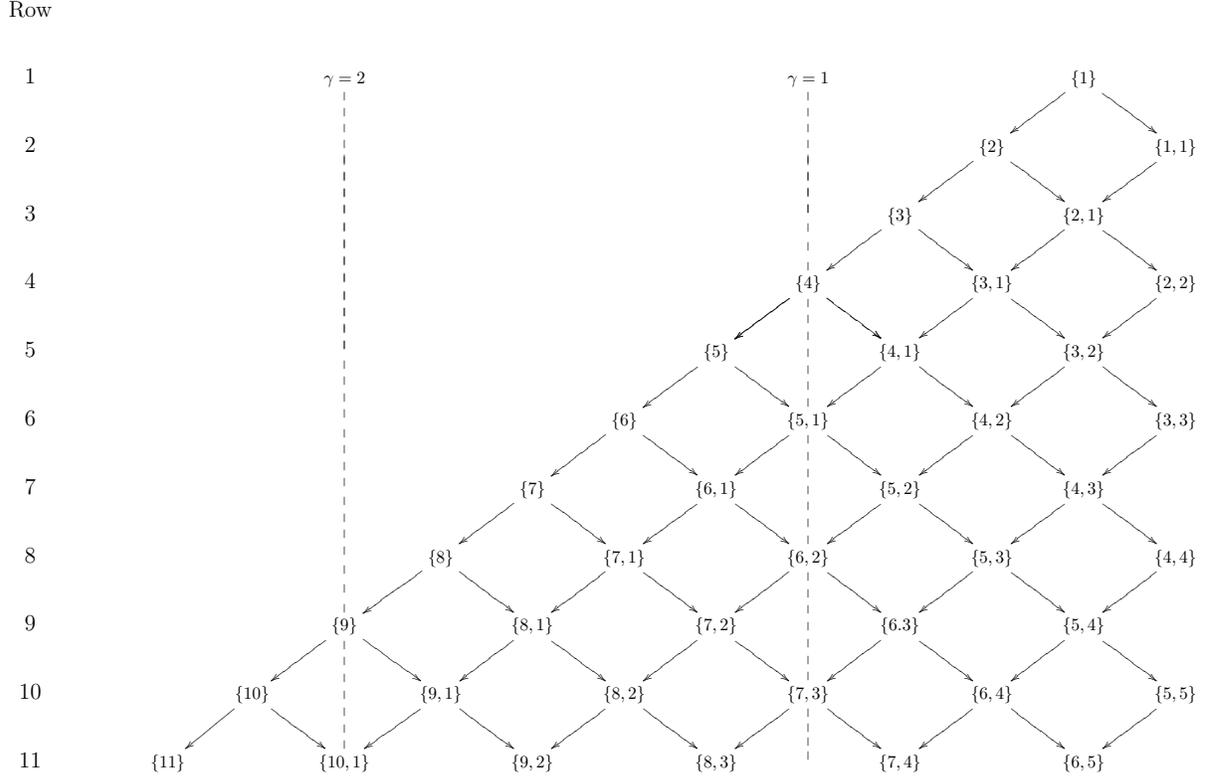
\begin{center}{\small  Figure 2:  First 11 rows of the Bratteli diagram $\Gamma$ of partitions with at most 2 parts.}\end{center}

Goodman and Wenzl \cite{GW} have used  $\Gamma$ to study tensor powers of the $\UU_q(\mathfrak{sl}_2)$-module $(\mathbb{C}^2)^{\otimes k}$
and the centralizing action of a certain Temperley-Lieb algebra $\TL_k(\xi_0)$, for the parameter $\xi_0$ that
depends on an $n$th root of unity $q$;  more specifically,
$\xi_0^{2} = (q^{\half}+q^{-\half})^2$. They identify ``critical lines'' in $\Gamma$, as those where $\beta_1-\beta_2 + 1 = n\,
(\modd n)$.
As it turns out,  the nodes on the critical lines correspond to the simple modules $\VV(n,s)$ on
$\Gamma_n$ labeled by $(n,s)$ for some $s \in \ZZ_n$.  In the  Bratteli diagrams $\Gamma_n$ (for $n=5$) and $\Gamma$  
displayed in Figure 1 and Figure 2, we are seeing just the first two of these critical lines, which are the vertical dashed lines corresponding to
$\gamma_1$ and $\gamma_2$. 

We assume that the partition $\beta = \{\beta_1,\beta_2\} \vdash k$ is fixed and write
\begin{equation}\label{eq:beta} \beta_1 -\beta_2 + 1 = \gamma n + \delta, \quad \text{where} \; \; 0 \le \delta \le n-1.\end{equation}
We will consider several cases that depend on the expression in \eqref{eq:beta}.
\smallskip

\noindent \textbf{Case 1:} \; $\gamma = 0$ so that $1 \le \beta_1 -\beta_2+1 = \delta \le n-1$  \smallskip

 In this case, 
 $$\{\beta_1, \beta_2\}  \leftrightarrow (\beta_1-\beta_2+1,\beta_2)$$ defines a one-to-one correspondence
between such partitions of $k$ to the right of the first critical line (the line where $\gamma = 1$) in $\Gamma$, and the simple $\Df_n$-modules
to the right of the $\gamma = 1$ line of the Bratteli diagram $\Gamma_n$.  

For example, when $n=5$ and $\beta = \{5,2\} \vdash 7$,  then $\beta_1-\beta_2+1 = 4 = 0\cdot 5 +4$, and we have
$\{5,2\} \; \leftrightarrow \; (4, 2)  \; \leftrightarrow \;  \VV(4,2)$. \smallskip

\noindent \textbf{Case 2:} \; $\gamma > 0$ and $\delta = 0$ so that $1 \le \beta_1 -\beta_2+1 = \gamma n$ \smallskip

Here $\beta$ is on the $\gamma$th critical line of $\Gamma$, and $\beta = \{\beta_1,\beta_2\} \leftrightarrow (n,\beta_2)$.

For example, when $n = 5$, and $\beta = \{6,2\} \vdash 8$,  then  $\beta_1 -\beta_2 + 1 = 5$, so $\gamma = 1$ and $\delta = 0$,
and $\{6,2\} \; \leftrightarrow \; (5,2)  \; \leftrightarrow \; \VV(5,2)$.  \smallskip

\noindent \textbf{Case 3:} \; $\gamma > 0$ and $\delta \neq  0$  \smallskip

In this case, $\delta$ measures how far $\beta = \{\beta_1,\beta_2\}$ is to the left of the $\gamma$th critical line.  
Then $\beta = \{\beta_1,\beta_2\}\, \gammarrow \, \{\beta_1-\delta, \beta_2+\delta\}$ reflects the partition from
the left of the $\gamma$th critical line to the right of it.     

For example,  when $\beta = \{11\}$  we have $11 -0 +1 = 12 = 2 \cdot 5 + 2$, so $\gamma = 2$, and $\delta = 2$.
Then $\{11\} \, \twoarrow \, \{9,2\}$ is a reflection about the second ($\gamma = 2$) critical line.    Since $9-2+1 = 8 = 1 \cdot 5 + 3$,  we need to do another reflection,
this time corresponding to the critical line $\gamma = 1$.  
So  $\{11\} \, \twoarrow \{9,2\} \, \onearrow \, \{6,5\}$.    Now with $\{6,5\}$, we are to the right of the first critical line,
and we know from Case 1,  that $\{6,5\}\; \leftrightarrow \; (6-5+1,5) = (2,5) \leftrightarrow \VV(2,5)$.      

Tensoring a projective module with a finite-dimensional module always gives a sum of
projective indecomposable modules.    The region to the left of the $\gamma=1$ critical line in $\Gamma_n$ has nodes labeled by projective modules.
So subsequent tensoring of them with $\VV(2,0)$  will yield only projective modules.     Since $\{11\}$ is to the left of the
second critical line,  it corresponds to a projective module, and since $\{9,2\}$ is between the first and second critical lines,
it also corresponds to a projective module.       Which ones?

We have $\{9,2\} \; \leftrightarrow \; \overline{(2,5)}  \; \leftrightarrow \;  \Pf(2, 5)$  (recall
labels in the second component are in $\ZZ_n (=\ZZ_5)$ here).     Now
\emph{projective} modules on opposite sides of critical lines  have complementary labels.  
Thus, $\{11\}  \; \leftrightarrow \; \overline{(3, 5+2)} = \overline{(3,7)}  \; \leftrightarrow \;  \Pf(3, 7)$.

Here is what the two 11th rows look like under the correspondence described above.  The dashed lines indicate the critical lines.   Some partitions
and module labels fall on the critical lines.
$${\small \hspace{-.8cm}\begin{array}{ccccccccccc}
   &  & { \bf  |}  & & &  {\bf  |} &&  \\
\Gamma:& \{11\} & \{10,1\} &\{9,2\}  & \{8,3\} & { \bf |} & \{7,4\}& \{6,5\} \\
  &   & { \bf  |} & &  &  {\bf |}& &  \\
\Gamma_5:& \overline{(3,7)} & (5,6) &\overline{(2,5)}&\overline{(4,4)} &{ \bf |}  & (4,4) &(2,5)\\
   \end{array}}$$
This example illustrates several phenomena worth noting.   Complementary projective modules in $\Gamma_n$ are ones with
the same composition factors, i.e. $\Pf(\ell,s)$ and $\Pf(n-\ell, \ell+s)$.  
There can exist
projective modules that occur in a row of $\Gamma_n$,  for which the corresponding simple module does not occur
in the same row.  The projective module  $\Pf(3,7)$  is an example of that.

 \subsection{Computing the dimension of the centralizer algebra}\label{S5.3}
 
Our objective here is to develop a formula for the dimension of the $k$th centralizer algebra $\End_{\Df_n}(\VV^{\ot k})$ for all $k \ge 1$, where
$\VV$ is any simple two-dimensional $\Df_n$-module $\VV(2,r)$.   The result is independent 
of the value of  $r$,  and it will be a consequence of the following lemma, which also holds for any algebraically closed field of characteristic zero.
 
\begin{lemma}\label{lem:dimhom}  For $1 \le \ell, \ell' \le n$ and $s,s' \in \ZZ_n$, the following hold:
 \begin{align}  \dimm_{\kk}\,\Hom_{\Df_n}\left(\VV(\ell,s), \VV(\ell',s')\right)\; &= \; \delta_{(\ell,s),(\ell',s')} \label{eq:dim1}\\
  \dimm_{\kk} \Hom_{\Df_n}\left(\VV(\ell,s), \Pf(\ell',s')\right)\; &= \; \delta_{(\ell,s),(\ell',s')} \label{eq:dim2} \\
 \dimm_{\kk}\Hom_{\Df_n}\left(\Pf(\ell,s), \VV(\ell',s')\right) \; &= \; \delta_{(\ell, s),(\ell',s')}\label{eq:dim3} \\
\dimm_{\kk}\Hom_{\Df_n}\left(\Pf(\ell,s), \Pf(\ell',s')\right) \; &= \;
 \begin{cases} 2 &\text{if} \; (\ell',s') = (\ell,s) \; \text{or}\; (n-\ell,s+\ell), \\
0  \, &\text{otherwise}. \end{cases}\label{eq:dim4}
\end{align}\end{lemma}
\begin{proof} The fact that $ \dimm_{\kk} \Hom_{\Df_n}\left(\VV(\ell,s), \VV(\ell,s')\right) = \delta_{(\ell,s),(\ell',s')}$ follows from 
simplicity of the modules and Schur's lemma.   
The assertion in \eqref{eq:dim2} holds because $\Pf(\ell',s')$, $\ell'\ne n$, has
a unique simple submodule, $\VV(\ell',s')$, which is its socle. 
The module $\Pf(\ell,s)$, $\ell \ne n$,  has a unique maximal submodule,  and a unique
simple quotient $\VV(\ell,s)$, which implies \eqref{eq:dim3}.  
Finally, \eqref{eq:dim4} is a direct consequence of the fact that $\Pf(\ell,s)$ has only  the simple modules
$\VV(\ell,s)$ and $\VV(n-\ell,s+\ell)$ as composition factors, each with multiplicity $2$,  
and  then Proposition \ref{Hom} below can be applied to derive the result. \end{proof}
\begin{proposition}{\rm (See \cite[Prop. 9.2.3]{EG}.)} \label{Hom} For a finite-dimensional algebra $\Hf$, let $\mathsf{S}$ be a simple module whose projective cover is $\mathsf{P}$.
Then for any  finite-dimensional $\Hf$-module $\mathsf{N}$,
\begin{align*}
\dimm_{\kk}\mathsf{Hom}_{\Hf}(\mathsf{P},\mathsf{N})=[\mathsf{N}:\mathsf{S}], 
\end{align*}
the multiplicity of $\mathsf{S}$ in a Jordan-H\"older series of $\mathsf{N}$.
\end{proposition}

Suppose now that $\Sf_1,\Sf_2,\dots, \Sf_{n^2}$ is a listing of the simple $\Df_n$-modules, and
$\Pf_1, \Pf_2, \dots, \Pf_{n^2}$ are respectively their projective covers.    Let $\Ir_k$ be the set of all $i \in \{1,2,\dots,n^2\}$
such that $\Sf_i$ or $\Pf_i$ or both occur in $\VV^{\ot k}$.  For $i \in \Ir_k$, let $s_i$ (resp. $p_i$) be the multiplicity of
$\Sf_i$ (resp. $\Pf_i$) in $\VV^{\ot k}$. Hence, not both $s_i$ and $p_i$ are zero when $i \in \Ir_k$.     
\textit{Because the module $\VV(n,s)$ is both simple and projective,  when $\VV(n,s)$ occurs as a summand of $\VV^{\ot k}$
 for some $s \in \ZZ_n$,   we will assume that $p_i = 0$ and $s_i$ is the multiplicity of $\VV(n,s)$ in $\VV^{\ot k}$.}   We adopt one further
convention:  When $\Pf_i = \Pf(\ell,s)$ for some $\ell \ne n$, and $p_i$ is its multiplicity in $\VV^{\ot k}$, then we
let $\Pf_i'$ be $\Pf(n-\ell, s+\ell)$ and use
$p_i'$ to denote the  multiplicity of $\Pf_i'$ in $\VV^{\ot k}$.   Thus, $\Pf(n-\ell, s+\ell)$ is $\Pf_j$ for some $j$,  and
it is also denoted as $\Pf_i'$ for $\Pf_i = \Pf(\ell,s)$.   When $k = 1$ in the next result, then $\Ir_k = \{1\}$, there is a unique summand
$\Sf_1=\VV$, and $s_1 = 1$ is the only nonzero term in \eqref{eq:dimcent}.
 
\begin{theorem}\label{thm:dimCent}  For any two-dimensional simple $\Df_n$-module $\VV$, the dimension of the centralizer algebra $\End_{\Df_n}\left(\VV^{\ot k}\right)$ for
$k \ge 1$ is given by
\begin{equation}\label{eq:dimcent}\dimm_{\kk} \End_{\Df_n}\left(\VV^{\ot k}\right)
= \sum_{i \in \Ir_k} p_i^2 + \sum_{i \in \Ir_k} (s_i+p_i)^2 +  2 \sum_{i \in \Ir_k} p_ip_i'.\end{equation}
\end{theorem}

\begin{proof} It follows from the decomposition of $\VV^{\ot k}$ into simple and indecomposable projective summands that
\begin{align*} \End_{\Df_n}\left(\VV^{\ot k}\right)  \cong  &\bigoplus_{i \in \Ir_k} \Hom_{\Df_n}\left(\Sf_i^{\oplus s_i}, \Sf_i^{\oplus s_i} \right)
 \bigoplus \;\, {\bigoplus}_{i \in \Ir_k} \Hom_{\Df_n}\left(\Sf_i^{\oplus s_i}, \Pf_i^{\oplus p_i} \right) \\
& \bigoplus \;\, \bigoplus_{i \in \Ir_k} \Hom_{\Df_n}\left(\Pf_i^{\oplus p_i}, \Sf_i^{\oplus s_i} \right) 
 \bigoplus  \;\, \bigoplus_{i \in \Ir_k} \Hom_{\Df_n}\left(\Pf_i^{\oplus p_i}, \Pf_i^{\oplus p_i} \right)\\  
& \bigoplus \bigoplus_{i \in \Ir_k} \Hom_{\Df_n}\left(\Pf_i^{\oplus p_i}, \Pf_i'^{\oplus p_i'} \right),   
\end{align*}
where if $\Pf_i = \Pf(\ell, s)$, then $\Pf_i' = \Pf(n-\ell, s+\ell)$.  Counting dimensions,  we have
\begin{align*} 
\dimm_{\kk} \End_{\Df_n}\left(\VV^{\ot k}\right) &= \sum_{i \in \Ir_k} \big(s_i^2 + s_i p_i + p_i s_i + 2p_i^2 
+ 2 p_i p_i'\big)\\
&= \sum_{i \in \Ir_k} p_i^2 +  \sum_{i \in \Ir_k} (s_i + p_i)^2 +  2\sum_{i \in \Ir_k} p_i p_i'.
\end{align*}  
\end{proof}

\begin{example} 
\label{Brat5} Pictured below are Rows 1-11 of the Bratteli diagram $\Gamma_5$, without the directed edges, but with the multiplicities displayed
as subscripts.   The multiplicities are the number of paths from $(2,0)$ at the top of the diagram to a particular vertex.  
Recall that $(\ell,s)$ is shorthand for $\VV(\ell,s)$ and $\overline{(\ell,s)}$ is
shorthand for $\Pf(\ell,s)$, where the second component should be interpreted modulo $n=5$ in this example.
{\small
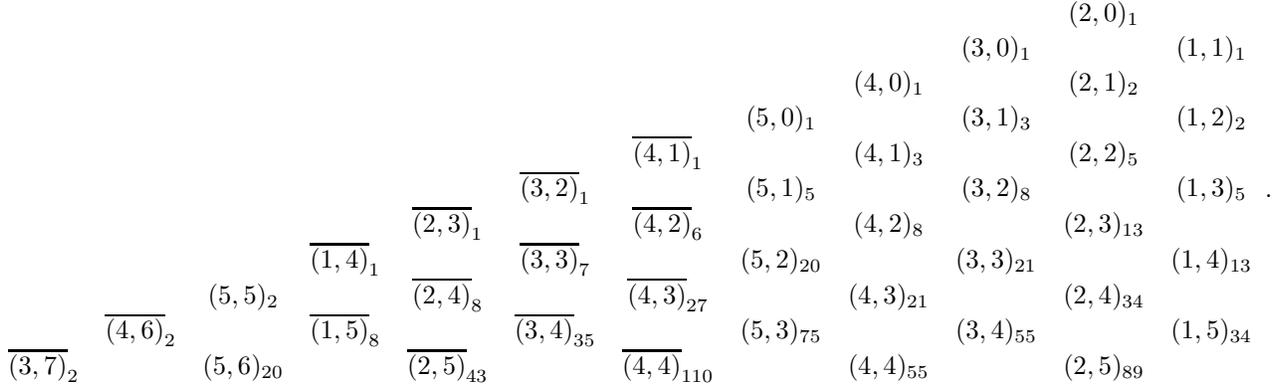
$$\hspace{-.25cm}\begin{array}{cccccccccccc}
&&&&&&&&&&(2,0)_{1}\\
&&&&&&&&&(3,0)_{1}&&(1,1)_1 \\
&&&&&&&&(4,0)_{1}&&(2,1)_2& \\
&&&&&&&(5,0)_1&& (3,1)_3 && (1,2)_2\\
&&&&&&\overline{(4,1)}_{1}&&(4,1)_3&&(2,2)_5 \\
&&&&&\overline{(3,2)}_1 && (5,1)_{5}&&(3,2)_8&&(1,3)_5 \\
&&&&\overline{(2,3)}_1&&\overline{(4,2)}_6 && (4,2)_8&& (2,3)_{13}\\
&&& \overline{(1,4)}_1 && \overline{(3,3)}_7 && (5,2)_{20} && (3,3)_{21} && (1,4)_{13} \\
&& (5,5)_2 & & \overline{(2,4)}_8& &\overline{(4,3)}_{27} &&(4,3)_{21} && (2,4)_{34} \\
&\overline{(4,6)}_2 && \overline{(1,5)}_8& &\overline{(3,4)}_{35} &&(5,3)_{75} && (3,4)_{55} && (1,5)_{34} \\
\overline{(3,7)}_2 &&(5,6)_{20} & & \overline{(2,5)}_{43}& &\overline{(4,4)}_{110} &&(4,4)_{55} && (2,5)_{89} \\
\end{array}.$$     } \medskip

\noindent 
In Table ~\ref{eq:Comp}, we compare the Catalan number $\mathcal{C}_k$ with $\dimm_{\kk}\End_{\Df_n}\left(\VV^{\ot k}\right)$ for $n=5$.  
\begin{center} \begin{tabular}[t]{|c||c|c|c|c|c|c|c|c|}
\hline \hline
        $k$  & $\mathcal{C}_k$ &  $\sum_{i \in \Ir_k} p_i^2$  & $\sum_{i \in \Ir_k} (s_i+p_i)^2$ & $2\sum_{i \in \Ir_k} p_ip_i' $ &
        $\dimm_{\kk} \End_{\Df_n}\left(\VV^{\ot k}\right)$ \\
	\hline  \hline 
	1 & 1 & 0 & $1^2$ & 0 &1 \\ \hline
	2 & 2 & 0 & $1^2+1^2$ & 0 &2 \\ \hline
	3 & 5  &0 & $1^2+2^2$ & 0 & 5 \\ \hline
	4 & 14  &0 & $1^2+3^2+2^2$ & 0 & 14 \\ \hline
	5 & 42  &$1^2$ & $4^2+5^2$ & 0 & 42\\ \hline
	6 & 132  &$1^2$ & $9^2+5^2+5^2$ & 0 & 132\\ \hline
	7 & 429  &$1^2+6^2$ & $(14)^2+(14)^2$ & 0 & 429\\ \hline
	8 & 1430 & $1^2 +7^2$ & $(20)^2 + (28)^2 + (14)^2$ & 0 &1430  \\ \hline
	9 & 4862& $8^2 +(27)^2$ & $2^2 + (42)^2 + (48)^2$ & 0 & 4865 \\ \hline
	10 &16796& $2^2 +8^2 +(35)^2$ & $(75)^2 + (90)^2 + (42)^2$ & $2\cdot2\cdot8 + 2 \cdot8 \cdot2$ & 16846 \\ \hline
	11 &58786& $2^2 +(43)^2 +(110)^2$ & $(20)^2 + (165)^2 + (132)^2$ & $2\cdot2\cdot43 + 2 \cdot43 \cdot2$ & 59346 \\ \hline
	\hline 
\end{tabular}
\smallskip
\captionof{table}{\emph{Catalan number $\mathcal{C}_k$ and $\dimm_{\kk} \End_{\Df_n}\left(\VV^{\ot k}\right)$ for $n=5$ and $1 \leq k \leq 11$.}  \label{eq:Comp}} \end{center}
\end{example}

\begin{remark}  The last summand $2 \sum_{i \in \Ir_k} p_ip_i'$ does not occur for $1 \le k \le 2n-1$.  Row $k=2n-2$ is the first time a module of the form $\Pf(1,s)$ appears
in $\Gamma_n$ for some $s \in\ZZ_n$,
and when tensored with $\VV$, it produces 2 copies of $\VV(n,s+1)$ on Row $2n-1$. Going down the left-hand diagonal of the Bratteli diagram
$\Gamma_n$, all multiplicities have been equal to 1  prior to this
shift at Row $2n-1$ to multiplicity 2.   The difference $2^2 - 1^2 = 3$ accounts for the fact that $4865 = 4862+3$ when $n=5$ and $k=9$.  \end{remark}

Some further observations are recorded in the next lemma.
\begin{lemma}\label{observations}
For the top $2n-2$ rows of $\Gamma$ and $\Gamma_n$, the following are true:
\begin{itemize}
\item[{\rm (a)}]  The first $n-1$ rows of $\Gamma$ and $\Gamma_n$ are isomorphic as graphs, where vertices are identified in the order they appear.  More precisely, 
vertex $\beta = \{\beta_1,\beta_2\}\vdash k$  in $\Gamma$ corresponds to $(\ell,r)$ in Row $k$ of $\Gamma_n$  
under the correspondence $\ell = \beta_1-\beta_2+1$ and $r = \beta_2$.

\item[{\rm (b)}]  The first $2n-2$ rows of $\Gamma$ and $\Gamma_n$ differ exactly at the diamonds lying on the first critical line.
These diamonds have as their top vertices $\{n-1\}$ in $\Gamma$ and $(n,0)$ in $\Gamma_n$  and are referred to as \emph{irregular diamonds} of $\Gamma_n$. 
An irregular diamond in $\Gamma_n$ has no edges in the northeast corner and double edges on the southwest corner.

\item[{\rm (c)}] In the first $2n-2$ rows, if both $(\ell,r)$ and $\overline{(\ell,r)}$ occur in Row $k$, then they appear as reflections across the first critical line. 
Moreover, in such a row, if $\overline{(\ell,r)}$ occurs, then $(\ell,r)$ must occur.
\end{itemize}
  \end{lemma}

\subsection{Comparing multiplicities of summands in $\Gamma$ and $\Gamma_n$.}\label{S5.4}

Since our arguments in this section will involve using multiplicities in various rows of the Bratteli graphs $\Gamma$ and $\Gamma_n$, we will adopt the following notational
conventions: 

Recall that the first (rightmost) critical line in $\Gamma_n$ (the one corresponding to $\gamma = 1$) separates the vertices corresponding to simple summands  and the vertices corresponding to projective summands.
In Row $k$, starting from the first critical line and moving to the left, the path counts for the vertices corresponding to the projective modules
will be recorded here as $p_1^{(k)}, p_2^{(k)}, p_3^{(k)}, \dots$, from \emph{right to left} (see the picture below). 
From the first critical line  moving to the right, the path counts will be recorded as $s_1^{(k)}, s_2^{(k)}, s_3^{(k)}, \dots$, from \emph{left to right}. 
If the first critical line cuts through a vertex in Row $k$,  the path count for that vertex is recorded as $s_0^{(k)}$, and we set $p_0^{(k)}=0$.
Otherwise both $s_0^{(k)}$ and $p_0^{(k)}$ are assumed to be 0.  When the same simple module
occurs in a later row, say in row $\ell > k$ for some $\ell$,  its path count will have the label $s_j^{(\ell)}$ (with the same subscript).  Analogously,  $p_j^{(k)}$ records the multiplicity of its projective cover, and $s_0^{(k)}$ records an occurrence of a module $\VV(n,s)$, which is both simple and projective among the summands.   
Furthermore, when
a summand does not occur in a certain row, we assume the corresponding value $p_j^{(k)}$ or $s_j^{(k)}$ is 0. In particular,  since there are no projective summands for $\VV^{\ot k}$ when $1 \le k \le n-1$, 
it follows that $p_j^{(k)}=0$ for all $j$ and all $1\leq k \leq n-1$.
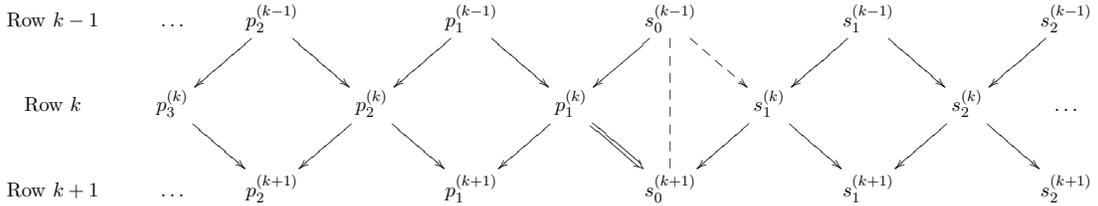
\ytableausetup{smalltableaux}
\begin{center}
\scalebox{0.7}{
\xymatrix{
\text{Row }k-1 &\dots & p_2^{(k-1)} \ar[dr]\ar[dl] && p_1^{(k-1)}\ar[dr]\ar[dl] & & s_0^{(k-1)} \ar@{--}[dd] \ar[dl]\ar@{-->}[dr] & & s_1^{(k-1)} \ar[dl]\ar[dr]& & s_2^{(k-1)} \ar[dl]\\
\text{Row }k & p_3^{(k)} \ar[dr] && p_2^{(k)} \ar[dl]\ar[dr] & & p_1^{(k)} \ar[dl]\ar@{=>}[dr]& & s_1^{(k)}\ar[dl]\ar[dr]& & s_2^{(k)} \ar[dl]\ar[dr] &\dots\\
\text{Row }k+1 &\dots &p_2^{(k+1)}&& p_1^{(k+1)}   & & s_0^{(k+1)} & & s_1^{(k+1)}  & & s_2^{(k+1)}
}} 
\end{center}

We set up the notation for path counts in $\Gamma$ in a similar fashion denoting ones on Row $k$ to the right of the first critical line by
 $\wt{s}_i^{(k)}$ from \emph{left to right}, and ones to the left of the first critical line starting with $\wt{p}_1^{(k)}$, $\wt{p}_2^{(k)}$, and proceeding \emph{right to left}.

 Recall that in Part (b) of Lemma~\ref{observations}, in Rows $1$ through $2n-2$, the vertices of $\Gamma$ and $\Gamma_n$ can be identified, so that the edges only differ at the irregular diamonds. This labeling of path counts depends on $n$, even though the underlying graph $\Gamma$ remains the same.

The following is key in our dimension argument.
\begin{proposition}\label{identifiedpathcount}   
For $1\leq k\leq 2n-2$, 
\begin{itemize}
\item[\rm {(a)}]  ${p}_i^{(k)}=\wt{p}_i^{(k)}$,
\item[\rm {(b)}]  ${p}_i^{(k)}+s_i^{(k)}=\wt{s}_i^{(k)}$,
\end{itemize}
where the index $i$ runs over all subscripts that occur in Row $k$.
\end{proposition}
\begin{proof}
The first $n-1$ rows of $\Gamma$ and $\Gamma_n$ are isomorphic as graphs.  Therefore,  the assertions in Proposition \ref{identifiedpathcount} are straightforward in light of the fact that $p_i^{(k)}=0$ for $1\leq k\leq n-1$, and direct summands
that are projective, but not simple,  occur only after Row $n-1$.

For  $n\leq k\leq 2n-2$,  we proceed by induction on $k$, where the base case $k=n$ is as follows:
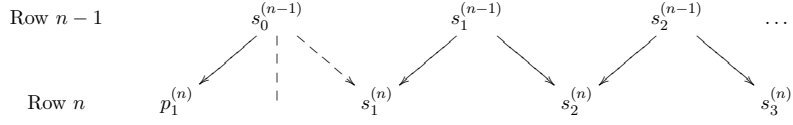
\begin{center}
\scalebox{0.7}{
\xymatrix{
\text{Row }n-1  & & s_0^{(n-1)} \ar@{--}[d] \ar[dl]\ar@{-->}[dr] & & s_1^{(n-1)} \ar[dl]\ar[dr]& & s_2^{(n-1)} \ar[dl]\ar[dr] & \dots \\
\text{Row }n &  p_1^{(n)}  & & s_1^{(n)} & & s_2^{(n)}  &   & s_3^{(n)}
}} 
\end{center}
By applying Lemma~\ref{observations} to both $\Gamma$ and $\Gamma_n$, we have $p_1^{(n)}=s_0^{{(n-1)}}=\wt{s}_0^{(n-1)}=\wt{p}_1^{(n)}$, where the middle equality holds by the statement for $k=n-1$, discussed in the semisimple case (i.e., when $1 \le k \le n-1$).

By a similar argument, $p_1^{(n)}+s_1^{(n)}=s_0^{(n-1)}+s_1^{(n-1)}=\wt{s}_0^{(n-1)}+\wt{s}_1^{(n-1)}=\wt{s}_1^{(n)}$, where the last equality is applying Lemma~\ref{observations} to $\Gamma$ at $\wt{s}_1^{(n)}$ (notice the dashed arrow above is replaced by a solid arrow in $\Gamma$). To see that (b) is true for $i>1$, observe that $p_i^{(n)}=0$ for $i>1$, and $s_i^{(n)}$ has two solid incoming arrows. Therefore,  $$s_i^{(n)}+p_i^{(n)}=s_i^{(n)}=s_{i-1}^{(n-1)}+s_{i}^{(n-1)}=\wt{s}_{i-1}^{(n-1)}+\wt{s}_{i}^{(n-1)}=\wt{s}_i^{(n)}.$$

We now proceed with the main induction step. Suppose statements (a) and (b) are true for $k=j-1$. We claim the statements are true for $k=j$ by a discussion of two cases.  

\textbf{Case 1:} Row $j$ has no vertices lying on the first critical line (i.e. no $s_0^{(j)}$ occurs).
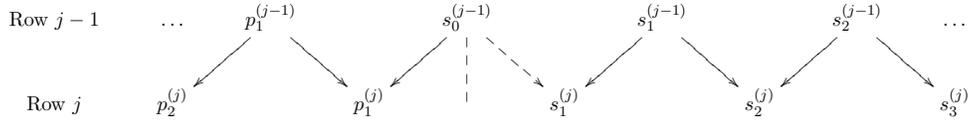
\begin{center}
\scalebox{0.7}{
\xymatrix{
\text{Row }j -1& \dots & p_1^{(j-1)}\ar[dl]\ar[dr]  & & s_0^{(j-1)} \ar@{--}[d] \ar[dl]\ar@{-->}[dr] & & s_1^{(j-1)} \ar[dl]\ar[dr]& & s_2^{(j-1)} \ar[dl]\ar[dr] & \dots \\
\text{Row }j & p_2^{(j)} &&  p_1^{(j)}  & & s_1^{(j)} & & s_2^{(j)}  &   & s_3^{(j)}
}} 
\end{center}
By the induction hypothesis, $p_0^{(j-1)}+s_0^{(j-1)}=\wt{s}_0^{(j-1)}$, but $p_0^{(j-1)}=0$ by definition, therefore  $s_0^{(j-1)}=\wt{s}_0^{(j-1)}$. Also $p_1^{(j-1)}=\wt{p}_1^{(j-1)}$ by induction hypothesis. It follows that $p_1^{(j)}={p}_1^{(j-1)}+s_0^{(j-1)}=\wt{p}_1^{(j-1)}+\wt{s}_0^{(j-1)}=\wt{p}_1^{(j)}$. For $i>1$, it is straightforward that $p_i^{(j)}={p}_i^{(j-1)}+p_{i-1}^{(j-1)}=\wt{p}_i^{(j-1)}+\wt{p}_{i-1}^{(j-1)}=\wt{p}_i^{(j)}$, because arrows to the left of the first critical line are all solid, between Rows $j-1$ and $j$. Now we turn to  the path counts to the right of the critical line. Notice $$p_1^{(j)}+s_1^{(j)}=p_1^{(j-1)}+s_0^{(j-1)}+s_1^{(j-1)}\stackrel{*}{=}s_0^{(j-1)}+\wt{s}_1^{(j-1)}=\wt{s}_0^{(j-1)}+\wt{s}_1^{(j-1)}=\wt{s}_1^{(j)},$$
where the equality marked by $*$ is given by the induction hypothesis for (b) when $k=j-1$ and $i=1$. Also $s_0^{(j-1)}=\wt{s}_0^{(j-1)}$ for the next equality, as proven earlier. The argument is easier for $i>1$ because the edges that are identified in $\Gamma_n$ and $\Gamma$ are away from the axis:
$$p_i^{(j)}+s_i^{(j)}=p_i^{(j-1)}+p_{i-1}^{(j-1)}+s_i^{(j-1)}+s_{i-1}^{(j-1)}=\wt{s}_i^{(j-1)}+\wt{s}_{i-1}^{(j-1)}=\wt{s}_i^{(j)}.$$

\textbf{Case 2:}  Row $j$ has a vertex lying on the first critical line (i.e. $s_0^{(j)}$ occurs): \begin{center}
\scalebox{0.7}
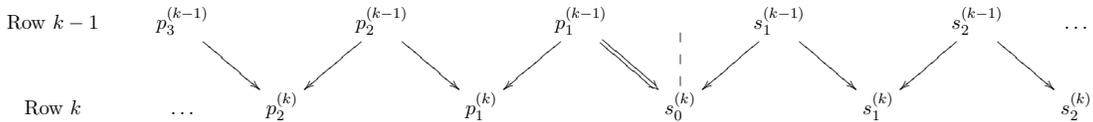
{
\xymatrix{
\text{Row }k-1 & p_3^{(k-1)} \ar[dr] && p_2^{(k-1)} \ar[dl]\ar[dr] & & p_1^{(k-1)}  \ar[dl]\ar@{=>}[dr] & \ar@{--}[d] & s_1^{(k-1)}  \ar[dl]\ar[dr]& & s_2^{(k-1)} \ar[dl]\ar[dr] &\dots\\
\text{Row }k &\dots &p_2^{(k)}&& p_1^{(k)}   & & s_0^{(k)} & & s_1^{(k)}  & & s_2^{(k)}
}} 
\end{center}
Here, (a) is true for all $i$ by an argument similar to that in Case 1, due to the fact all $p_i^{(k)}$ have two incoming arrows, which are identified with those edges in $\Gamma$. To see (b) is true for $i=0$, observe that
\begin{align*}
p_0^k+s_0^k=s_0^k=2p_1^{k-1}+s_1^{k-1}=p_1^{k-1}+\wt{s}_1^{k-1}=\wt{p}_1^{k-1}+\wt{s}_1^{k-1}=\wt{s}_0^k.
\end{align*}
The remaining multiplicities for (b) when $i>0$ are true  by an argument similar to that in Case 1.  This concludes the proof of Proposition \ref{identifiedpathcount}.
\end{proof}
In Theorem \ref{doublecentralizer}, we will use the path count comparisons in Proposition  \ref{identifiedpathcount} together with the result in Theorem \ref{thm:dimCent} to show that the Temperley-Lieb algebra 
$\TL_k(\xi)$ is indeed isomorphic to the centralizer algebra $\End_{\Df_n}(\VV(2,r)^{\ot k})$ for any $r \in \mathbb Z_n$ and $1 \le k \le 2n-2$.

\begin{theorem}
\label{doublecentralizer} 
Assume $n \ge 2$,  and  $\VV = \VV(2,r)$ for any $r \in \mathbb Z_n$. Then the algebra homomorphism $\pi: \mathsf{TL}_k(\xi) \to \End_{\Df_n}(\mathsf{V}^{\otimes k})$ in Theorem~\ref{thm:TLalghom}
is an isomorphism for $1\leq k\leq 2n-2$.
 \end{theorem}
 
\begin{proof} 
It  follows immediately from the dimension count (see Proposition \ref{identifiedpathcount}) that
\begin{align*}
\dimm_{\mathbb C}\mathsf{TL}_k(\xi)=\sum_{i \in \Ir_k} \wt{s}_i^2+\sum_{i \in \Ir_k} \wt{p}_i^2=\sum_{i \in \Ir_k} (p_i+s_i)^2+\sum_{i \in \Ir_k}p_i^2 = \dimm_{\mathbb C}\End_{\Df_n}(\mathsf{V}^{\otimes k}),
\end{align*}
where the last equality is a consequence of Theorem \ref{thm:dimCent} and the fact that $p_ip_i' = 0$ for $i \in \Ir_k$ when 
$1 \le k \le 2n-2$. Since the map $\pi$ is injective for all $k$, it follows
by the above dimension count that it is also surjective, hence an isomorphism,  when $1\leq k\leq 2n-2$.  
\end{proof}

\subsection{Further questions} In Theorem \ref{thm:TLalghom}  of this paper, we have shown that for an arbitrary two-dimensional simple $\Df_n$-module $\VV = \VV(2,r)$,  there exists an injective algebra homomorphism $\pi: \mathsf{TL}_k(\xi) \rightarrow \End_{\Df_n}(\VV^{\ot k})$ for $\xi = -(q^{\half}+q^{-\half})$, and when $1 \le k \le 2n-2$, $\pi$ is an isomorphism according to Theorem~\ref{doublecentralizer}.
  \begin{enumerate}
  \item What generates the full centralizer algebra  $\End_{\Df_n}(\VV(2, r)^{\ot k})$ when $k \ge 2n-1$? 
  \item For arbitrary $r \in \ZZ_n$, can $\End_{\Df_n}(\VV(2,r)^{\ot k})$ be realized as a diagram algebra?
  \item Is the category of $\mathsf{D}_n$-modules a highest weight category, in the sense of \cite{CPS}? If so, can we give an explicit description of the standard, costandard and tilting modules as in \cite{RW2}?
  \item If the above is true, do the Jones-Wenzl projectors (see e.g., \cite{K}) provide projections onto each standard summand of $\VV(\ell,r)^{\otimes k}$?    
 \item  Do the $p$-Jones-Wenzl projectors of \cite{BLS} provide projections onto the tilting summands of $\VV(\ell,r)^{\otimes k}$?
  \item For other simple $\Df_n$-modules $\VV(\ell,r)$ with $\ell \ge 3$,  what is the decomposition of $\VV(\ell,r)^{\otimes k}$ into simple and
  projective summands, and what is the centralizer algebra $\End_{\Df_n}(\VV(\ell,r)^{\ot k})$? \end{enumerate}

\subsection*{Acknowledgments}  We thank the referee for helpful comments that  enabled us to generalize and shorten the paper.  Our work on this project began in connection with  the workshop WINART2 at  the University of Leeds in May 2019.  The authors would like to extend thanks 
to the organizers of WINART2 and  to the University of Leeds for its hospitality, and to acknowledge support
provided by a University of Leeds conference grant,  the London Mathematical Society Workshop Grant WS-1718-03,  the US National Science Foundation DMS 1900575, the Association for Women in Mathematics (NSF Grant DMS-1500481), and by a research fellowship from the Alfred P. Sloan Foundation.  Nguyen was supported by the Naval Academy Research Council.  Biswal gratefully acknowledges the Max Planck Institute for the fellowship she received as a postdoctoral researcher there and for providing an excellent atmosphere for research.  Benkart passed away on April 29, 2022.


\end{document}